\documentclass[11pt]{article}
\usepackage[centertags]{amsmath}
\usepackage{amsfonts}
\usepackage{amssymb}
\usepackage{amsthm,color}
\usepackage{newlfont}
\usepackage{bbm}

\newtheorem{Theorem}{Theorem}[section]

\newtheorem{Lemma}[Theorem]{Lemma}

\newtheorem{Remark}[Theorem]{Remark}

\newtheorem{Hypothesis}{Hypothesis}

\newtheorem{Hypothesisalt}{Hypothesis}[Hypothesis]

\newenvironment{Hypothesisp}[1]{
  \renewcommand\theHypothesisalt{#1}
  \Hypothesisalt
}{\endHypothesisalt}

\numberwithin{equation}{section}

\begin{document}

\def\le{\left}
\def\r{\right}
\def\cost{\mbox{const}}
\def\a{\alpha}
\def\d{\delta}
\def\ph{\varphi}
\def\e{\epsilon}
\def\la{\lambda}
\def\si{\sigma}
\def\La{\Lambda}
\def\B{{\cal B}}
\def\A{{\mathcal A}}
\def\L{{\mathcal L}}
\def\O{{\mathcal O}}
\def\bO{\overline{{\mathcal O}}}
\def\F{{\mathcal F}}
\def\K{{\mathcal K}}
\def\H{{\mathcal H}}
\def\D{{\mathcal D}}
\def\C{{\mathcal C}}
\def\M{{\mathcal M}}
\def\N{{\mathcal N}}
\def\G{{\mathcal G}}
\def\T{{\mathcal T}}
\def\R{{\mathbb R}}
\def\I{{\mathcal I}}

\def\bw{\overline{W}}
\def\phin{\|\varphi\|_{0}}
\def\s0t{\sup_{t \in [0,T]}}
\def\lt{\lim_{t\rightarrow 0}}
\def\iot{\int_{0}^{t}}
\def\ioi{\int_0^{+\infty}}
\def\ds{\displaystyle}
\def\pag{\vfill\eject}
\def\fine{\par\vfill\supereject\end}
\def\acapo{\hfill\break}

\def\beq{\begin{equation}}
\def\eeq{\end{equation}}
\def\barr{\begin{array}}
\def\earr{\end{array}}
\def\vs{\vspace{.1mm}   \\}
\def\rd{\reals\,^{d}}
\def\rn{\reals\,^{n}}
\def\rr{\reals\,^{r}}
\def\bD{\overline{{\mathcal D}}}
\newcommand{\dimo}{\hfill \break {\bf Proof - }}
\newcommand{\nat}{\mathbb N}
\newcommand{\E}{\mathbb E}
\newcommand{\Pro}{\mathbb P}
\newcommand{\com}{{\scriptstyle \circ}}
\newcommand{\reals}{\mathbb R}

\newcommand{\red}[1]{\textcolor{red}{#1}}

\def\Amu{{A_\mu}}
\def\Qmu{{Q_\mu}}
\def\Smu{{S_\mu}}
\def\H{{\mathcal{H}}}
\def\Im{{\textnormal{Im }}}
\def\Tr{{\textnormal{Tr}}}
\def\E{{\mathbb{E}}}
\def\P{{\mathbb{P}}}
\def\span{{\textnormal{span}}}
\title{Averaging principle for slow-fast systems of stochastic PDEs with rough coefficients}
\author{Sandra Cerrai\thanks{Partially supported by  NSF grants DMS-1712934 - {\em Analysis of Stochastic Partial Differential Equations with Multiple
Scales} and DMS-1954299 - {\em Multiscale Analysis of Infinite-Dimensional Stochastic Systems}}\\
\normalsize University of Maryland\\
\normalsize College Park, USA
\and
Yichun Zhu\\
\normalsize University of Maryland\\
\normalsize College Park, USA}
\date{}

\date{}
\maketitle

\begin{abstract}
In this paper, we consider a class of slow-fast systems of stochastic partial differential equations  where the nonlinearity in the slow equation is not continuous and unbounded. We first provide conditions that ensure the existence of a martingale solution. Then we prove that the laws of the slow motions are tight, and any of their limiting points is a martingale solution for a suitable averaged equation. Our results apply to systems of stochastic reaction-diffusion equations where the reaction term in the slow equation is only continuous and has polynomial growth.

\end{abstract}

\section{Introduction}
\label{sec1}

In this paper, we study the validity of an averaging principle for a class of slow-fast systems of stochastic partial differential equations. These systems are characterized by having weak regularity assumptions for the nonlinearity in the slow equation. Due to the weakness of these conditions, the authors are only able to prove the existence of martingale solutions and characterize the distributions of limiting points for the slow motions as solutions of a suitable averaged equation. After the publication of \cite{SCMF1} and \cite{SC1}, there have been many papers in recent years that have studied the validity of averaging principles for various types of slow-fast systems of SPDEs, but these previous papers have all assumed well-posedness in appropriate functional spaces, whereas the systems considered in this paper are so irregular that it is not  possible to have the existence and uniqueness of solutions, not even in the martingale sense, for either the slow-fast system or the limiting averaged equation.

We consider the following system of stochastic PDEs on a separable Hilbert space $H$, endowed with the scalar product $\langle \cdot, \cdot \rangle$ and the corresponding norm $\|\cdot \|$,
\begin{equation}\label{eq spde-intro}
\begin{aligned}
\begin{cases}
\ds{du_\e(t)= \left[A_1 u_\e(t) dt + F_1(t, u_\e(t), v_\e(t)) \right]\,dt + Q_1\,dW^{1}_t,\ \ \ \ \ u_{\e}(0)=u_0,}\\[10pt]
\ds{dv_\e(t)= \frac{1}{\e}\left[ A_2 v_\e(t) + F_2(u_\e(t), v_\e(t)) \right]\,dt + \frac{1}{\sqrt{\e}}\, Q_2\,dW^{2}_t,\ \ \ \ \ v_{\e}(0)=v_0.}
\end{cases}
\end{aligned}
\end{equation}
 The stochastic perturbations are given by two independent cylindrical Wiener processes in $H$,  which are white in time and colored in space, with covariance operators $Q_1^2$ and $Q_2^2$, respectively.
 The operators $A_1$ and $A_2$ are the generators of two analytic semigroups $S_1(t)$ and $S_2(t)$, respectively. 
 
 The non-linearity $F_2$ in the fast equation is Lipschitz continuous, but the non-linearity $F_1$ in the slow equation is  not assumed to be continuous nor  bounded or linearly bounded.    
More precisely, we assume that   $F_1:[0,T]\times D(F_1)\subseteq [0,T]\times H\times H \to H$ is measurable and there exists a family of bounded and measurable mappings $F^{\,\theta}_1$, depending on a parameter $\theta \in (0,1)$ and defined on $[0,T]\times H\times H$ with values in $H$, such that
\begin{enumerate}
\item[C1.] for all $t \in\,[0,T]$  and $h \in\,H$ the mapping 
\[(x,y) \in\,H\times H\mapsto \langle F^\theta_1(t,x,y),h\rangle \in\,\mathbb{R},\]
is continuous;
\item[C2.] there exists a convex and lower semicontinuous mapping $V:H\times H \to [1, \infty]$ such that
\[\| F^{\theta}_1(t,x,y) \|^2 \leq \| F_1(t,x,y) \|^2 \leq V(x,y),\ \ \ \ (x,y) \in\,H\times H,
\]
and
\[\| F_1(t,x,y)- F^{\theta}_1(t,x,y) \| \leq \theta\, V(x,y),\ \ \ \  (x,y) \in\,D(F_1)
\]
for all $t \in\,[0,T]$ and $\theta \in\,(0,1)$.
\end{enumerate}

We will show that these conditions are satisfied for example when $H=L^2(\mathcal{O})$, for some bounded and smooth domain $\mathcal{O}\subset \mathbb{R}^d$, and 
\[F_1(t,x,y)(\xi)=b(t,\xi,x(\xi),y(\xi)),\ \ \ \ \ t \in\,[0,T],\ \ \ \xi \in\,\mathcal{O},\]
for some  function $b:[0,T]\times \mathcal{O}\times \mathbb{R}^2\to\mathbb{R}$ which is just continuous and such that $b(t,\xi,\cdot):\mathbb{R}^2\to\mathbb{R}$ has polynomial growth and the following one-sided growth condition holds
\[\sup_{t \in\,[0,T]}b(t,\xi,\si+\rho,\la)\si\leq c\left(\alpha(\xi)+|\si|^2+|\la|^{\kappa_1}+|\rho|^{\kappa_2}	\right),\]
for all $(t,\xi) \in\,[0,T]\times \mathcal{O}$ and $(\la,\si,\rho) \in\,\mathbb{R}^3$ and for some positive constants $c, \kappa_1, \kappa_2$ and some function $\alpha \in\,L^1(\mathcal{O})$.  In this regard, it is worth noting that in \cite{SC2}, systems of reaction-diffusion equations with coefficients having polynomial growth were considered. However, in that case, the coefficients were assumed to be regular, stronger dissipativity conditions were imposed, and the coefficient of the slow equation had polynomial growth only in the slow variable, while the coefficient in the fast equation could have polynomial growth only in the fast variable. In contrast, in the example we consider in this paper, the reaction term in the slow equation is only continuous and has polynomial growth, both in the slow and in the fast variables.

Conditions C1. and C2.  are not new in the current literature. They have been introduced in several papers (see e.g. \cite{VBDPMR1}, \cite{VBDPMR2} and \cite{VBDPMR3}) in order to study the existence  of measure-valued solutions to the Fokker Planck equations associated with a wide class of stochastic differential equations both in finite and in infinite dimensional spaces.
Here, we are proving how the same kind of assumptions imply the existence of $C([0,T];H)$-valued martingale solutions for  the following  class of stochastic evolution equations in arbitrary separable Hilbert spaces
\[dX(t)=\left[A X(t)+F(t,X(t))\right]\,dt+Q\,dW_t,\ \ \ \ \ X(0)=x.\]

This result applies also to systems like \eqref{eq spde-intro}, once we rewrite them as
\[dX^\e(t)=\left[A^\e X^\e(t)+F^\e(t,X^\e(t))\right]\,dt+Q^\e\,dW_t,\ \ \ \ \ X^\e(0)=(u_0,v_0),\]
where  
\[A_\e:=(A_1, \e^{-1} A_2),\ \   F^\e:=(F_1,\e^{-1} F_2),\ \ Q^\e:=(Q_1,\e^{-1/2} Q_2),\]
and   where $W_t:=(W^1_t, W^2_t)$. This means that for every $(u_0,v_0) \in\,D_V:=\{V<\infty\}$ and $\e \in\,(0,1)$  there exists a martingale solution. That is there exists a stochastic basis, $({\Omega}^{\e}, {\mathcal{F}}^{\e}, \{{\mathcal{F}}_t^{\e}\}_{t \geq 0}, {\mathbb{P}}^{\e})$, an adapted cylindrical Wiener process ${W}^{\e}_t=(W^{1,\e}_t, W^{2,\e}_t)$ in $H\times H$, and an adapted process ${X}^{\e}=(u_\e,v_\e) \in\,L^2(\Omega^\e;C([0,T];D(F_\e))$ such that
\[X^{\e}(t) = S^{\e}(t)(u_0,v_0) + \int_0^t S^{\e}(t-s) F^{\e}(s, X^{\e}(s)) ds +\int_0^t S^{\e}(t-s) Q^{\e} d{W}^{\e}_s,\ \ \ \ t\geq 0,\] 
where $S^\e(t)$ is the semigroup generated by $A^\e$.

Next, we assume that there exists some constant $c>0$ such that for every initial condition $(u_0,v_0) \in\,D_V$
\begin{equation}
	\label{intro1}
\sup_{\e \in\,(0,1)}\int_0^T\mathbb{E}^{\e} V(X^{\e}(t))\,dt\leq c \,V(u_0,v_0),	
\end{equation}
and we prove that under this condition the family $\{\mathcal{L}(u_\e)\}_{\e \in\,(0,1)}$ is tight in $C([0,T];H)$.
Our purpose is showing that any weak limit point $\mu_{\text{\tiny{slow}}}$ of 
$\{\mathcal{L}(u_\e)\}_{\e \in\,(0,1)}$ in $C([0,T];H)$ is a martingale solution of a suitable limiting stochastic equation in $H$ whose non-linearity is obtained by averaging $F_1$ with respect to the invariant measure of the fast equation.

We would like to stress that the roughness of $F_1$ makes the definition of the averaged coefficient $\bar{F}_1$ quite delicate and requires some approximation procedures. Moreover, $\bar{F}_1$ inherits from $F_1$ the same roughness and the proof of the averaging limit requires several new ideas compared to what is available in the existing literature on the  averaging for slow-fast systems of SPDEs with regular coefficients.

\medskip

Before concluding this introduction,  we would like to describe the plan of the paper. After introducing in Section 2 all the notations and assumptions that will used throughout the paper, in Section 3, we describe the ergodic properties of the fast equation, define the average of the nonlinearity $F_1$ in terms of the invariant measure of the fast equation, introduce the corresponding averaged limiting equation, and state the main theorem of the paper. In Section 4, we prove several a-priori bounds for the solution of system \eqref{eq spde-intro} and we give a proof of the tightness of the laws of the solutions of the slow equation. In  Section 5, we prove the validity of the averaging limit, by using a localization- in-time argument. In Section 6, we provide an example of a class of systems of stochastic reaction-diffusion equations in bounded domains of $\mathbb{R}^d$, where the slow equation has a reaction coefficient that is only continuous and has polynomial growth, both in the slow and in the fast variable. Finally, in Appendix A, we give a proof of the existence of a $C([0,T];H)$-valued martingale solution for a general class of stochastic evolution equations with rough coefficients, that cover also the system \eqref{eq spde-intro}.

\section{Notations and assumptions}
\label{sec2}
Let $H$ be a separable Hilbert space, 
endowed with the scalar product $\langle \cdot, \cdot \rangle$ and the corresponding norm $\|\cdot \|$. We denote by $B_b(H)$ the Banach space of Borel bounded functions $\varphi:H\to \mathbb{R}$, endowed with the norm
\[\Vert \varphi\Vert_0:=\sup_{x \in\,H}\vert \varphi(x)\vert.\]
We denote by $C_b(H)$ the subspace of continuous functions. Moreover, we denote by $\text{Lip}_b(H)$ the subspace of all Lipschitz continuous functions. $\text{Lip}_b(H)$ is a Banach space, endowed with the norm
\[\Vert\varphi\Vert_{\text{\tiny{Lip}}}:=\Vert \varphi\Vert_0+[\varphi]_{\text{\tiny{Lip}}}:=\Vert \varphi\Vert_0+\sup_{\substack{x, y \in\,H\\x\neq y}}\frac{\vert\varphi(x)-\varphi(y)\vert}{\Vert x-y\Vert}.\]

In the present paper, we are dealing with the following system of stochastic equations in the space $H$
\begin{equation}\label{eq spde}
\begin{aligned}
\begin{cases}
\ds{du_\e(t)= \left[A_1 u_\e(t) dt + F_1(t, u_\e(t), v_\e(t)) \right]dt + Q_1\,dW^{1}_t,\ \ \ \ \ u_{\e}(0)=u_0,}\\[10pt]
\ds{dv_\e(t)= \frac{1}{\e}\left[ A_2 v_\e(t) + F_2(u_\e(t), v_\e(t)) \right]dt + \frac{1}{\sqrt{\e}}\, Q_2\,dW^{2}_t,\ \ \ \ \ v_{\e}(0)=v_0.}
\end{cases}
\end{aligned}
\end{equation}
Throughout the present section, we will introduce  the assumptions we make on the coefficients  and the noise and we will introduce some notations.
 
 \begin{Hypothesis}\label{h1}

\begin{enumerate}

\item
The operators $(A_1, D(A_1))$ and $(A_2, D(A_2))$ generate the analytic semigroups $S_1(t)$ and $S_2(t)$ in $H$, respectively. Moreover, there exist two complete orthonormal systems, $\{ e_{1,k} \}_{k \in \mathbb{N}}$ and $\{ e_{2,k} \}_{k \in \mathbb{N}}$ in $H$, and two non-decreasing sequences of positive real numbers $\{\a_{1,k}\}_{k \in \mathbb{N}}$ and $\{\a_{2,k}\}_{k \in \mathbb{N}}$, such that 
\[
A_1 e_{1,k}= -\a_{1,k} e_{1,k},\ \ \ \ A_2 e_{2,k}= -\a_{2,k} e_{2,k},\ \ \ \ k \in \mathbb{N}.\]

\item
The bounded linear operators $Q_1$ and $Q_2$ share the same set of eigenvectors with the operators $A_1$ and $A_2$, respectively. Namely, there exist two sequences of real numbers $\{\lambda_{1,k}\}_{k \in \mathbb{N}}$ and $\{\lambda_{2,k}\}_{k \in \mathbb{N}}$ such that 
\[
Q_1 e_{1,k}= \lambda_{1,k} e_{1,k},\ \ \ \ Q_2 e_{2,k}= \lambda_{2,k} e_{2,k},\ \ \ \ k \in \mathbb{N}.
\]
\item
There exist two strictly positive constants $\gamma_1$ and $\gamma_2$  such that
\begin{equation}
\label{reg}
\sum_{k\in \mathbb{N}} \lambda_{1,k}^2 \a_{1,k}^{2\gamma_1-1} <\infty,\ \ \ \ \sum_{k\in \mathbb{N}} \lambda_{2,k}^2 \a_{2,k}^{2\gamma_2-1} <\infty.	
\end{equation}
\end{enumerate}
\end{Hypothesis}
\begin{Remark}
{\em \begin{enumerate}
 \item 	IN view of Condition 1. in Hypothesis \ref{h1}, we have
\begin{equation}
\label{alpha}
\inf_{k \in\,\mathbb{N}}	\a_{1,\,k}=\a_{1, 1}>0,\ \ \ \ \ \ \inf_{k \in\,\mathbb{N}}	\a_{2,\,k}=\a_{2, 1}>0.
\end{equation}
\item Thanks to \eqref{alpha},  for every $\delta\geq 0$ we can define the norms
\[\Vert x\Vert_{D((-A_i)^\delta)}:=\Vert (-A_i)^\delta x\Vert,\ \ \ \ \ i=1, 2.\]
Notice that for every $\delta>0$ the space $D((-A_i)^\delta)$, endowed with the norm above, is compactly embedded in $H$, for $i=1, 2$.
 \end{enumerate}

}	
\end{Remark}

In what follows, we will denote by $\mathcal{H}_2$ the  product space $H \times H$,  endowed the inner product and the corresponding norm
\[\langle (x_1,y_1), (x_2,y_2)\rangle_{\mathcal{H}_2}:= \langle x_1, x_2 \rangle + \langle y_1, y_2 \rangle,\ \ \ \ \|(x,y)\|_{\mathcal{H}_2}:= \sqrt{\|x\|^2 + \|y\|^2}.\]
 Clearly, the family of vectors
\[
\Lambda:=\{(e_{1,j},0)\}_{j \in\,\mathbb{N}}\cup \{(0,e_{2,j})\}_{j \in\,\mathbb{N}}\]
is an orthonormal basis for the Hilbert space $\mathcal{H}_2$.

For every $\e>0$, we define the operators $A^{\e}$, $Q^{\e}$ by setting
\[
A^{\e}(x,y) = (A_1x, \e^{-1} A_2 y), \ \ \ \ (x,y) \in\,D(A_1)\times D(A_2)\subset \mathcal{H}_2,\]
and 
\[Q^{\e}(x,y) = (Q_1x, \e^{-1/2} Q_2y),\ \ \ \ (x,y) \in \mathcal{H}_2.
\]
It is immediate to check that $\Lambda$ is a set of eigenvectors which diagonalizes the operators $A^{\e}$ and $Q^{\e}$ simultaneously and $\{\a_{i,k}\}_{i=1,2,\  k \in \mathbb{N}}$ are eignevalues of $A^{\e}$, while $\{\lambda_{i,k}\}_{i=1,2,\ k \in \mathbb{N}}$ are eigenvalues of $Q^{\e}$. Due to \eqref{reg}
\[
\sum_{k \in \mathbb{N}} \lambda_{1,k}^2 \a_{1,k}^{2\gamma_1-1}+ \lambda_{2,k}^2 \a_{2,k}^{2\gamma_2-1} < \infty.
\] 
Moreover, if $S^{\e}(t)$ is the analytic semigroup generated by $A^{\e}$ on the space $\mathcal{H}_2$, we have
\[
S^{\e}(t)(x,y)= (S_1(t)\,x, S_2(t/\e)\,y),\ \ \ \ (x,y) \in\,\mathcal{H}_2.
\]

\medskip

Concerning the nonlinearity $F_1$, we shall assume the following conditions.
\begin{Hypothesis}\label{h2}
The mapping  $F_1:[0,T]\times D(F_1)\subseteq [0,T]\times \mathcal{H}_2 \to H$ is measurable. Moreover, there exists a family of mappings $\{F^{\,\theta}_1\}_{\theta \in (0,1)}$, defined on $[0,T]\times \mathcal{H}_2$ with values in $H$, such that the following conditions hold.
\begin{enumerate}
\item
For every $\theta \in\,(0,1)$, the mapping $F^{\theta}_1:[0,T]\times \mathcal{H}_2\to H$ is measurable and bounded. Moreover, for all $t \in\,[0,T]$  and $h \in\,H$ the mapping 
\[(x,y) \in\,\mathcal{H}_2\mapsto \langle F^\theta_1(t,x,y),h\rangle \in\,\mathbb{R},\]
is continuous.
\item
There exists a convex and lower semicontinuous mapping $V:\mathcal{H}_2 \to [1, \infty]$, with $D_V:=\left\{V<\infty\right\}\subseteq D(F_1)$, such that for all $\theta \in\,(0,1)$ and $t \in\,[0,T]$
\begin{equation} \label{sy400}
\| F^{\theta}_1(t,x,y) \|^2 \leq \| F_1(t,x,y) \|^2 \leq V(x,y),\ \ \ \ (x,y) \in\,\mathcal{H}_2,
\end{equation}
and
\begin{equation} \label{sy33}
\| F_1(t,x,y)- F^{\theta}_1(t,x,y) \| \leq \theta\, V(x,y),\ \ \ \  (x,y) \in\,D(F_1).
\end{equation}
\end{enumerate}
\end{Hypothesis}

\begin{Remark}
{\em \begin{enumerate}
 \item When $(x,y) \notin\,D(F_1)$, we take $\Vert F_1(t,x,y)\Vert=+\infty$. This allows to interpret \eqref{sy400}, for every $(x,y) \in\,\mathcal{H}_2.$
\item In Hypothesis \ref{h1}, we have assumed that $\alpha_{1, 1} >0$. However, this is not necessary. Actually if we define  $\hat{A}_1:=A_1-\delta I$, and $\hat{F}_1(t,\cdot)=F_1(t,\cdot)+\delta I$, for some $\delta>-\a_{1,1}$, we have that 
\begin{equation}
\hat{A}_1x + \hat{F}_1(t,x,y)=A_1x+F_1(t,x,y),\ \ \ \ \ x \in\,D(A),\ \ (x,y) \in\,D(F_1),\ \ \ t \in\,[0,T],
\end{equation}
and $\hat{A}_1$ is a negative operator. Now, for every $\theta \in\,(0,1)$ we define
\[\hat{F}_1^\theta(t,x,y):=F_1^\theta(t,x,y)+\frac{\delta\,x}{1+\theta\ \Vert x\Vert},\ \ \ \ (t,x,y) \in\,[0,T]\times \mathcal{H}_2,\]
we have that $\hat{F}_1^\theta$ is measurable and bounded and for every $h \in\,H$ and $t \in\,[0,T]$ the mapping $(x,y) \in\,\mathcal{H}_2 \mapsto \langle F_1^\theta(t,x,y),h\rangle \in\,\mathbb{R}$ is  continuous. Moreover,
\[\Vert \hat{F}_1^\theta(t,x,y)\Vert^2\leq 2\,\Vert F_1^\theta(t,x,y)\Vert^2+\delta^2\Vert x\Vert^2.\]
Therefore, if we define 
\begin{equation}
\hat{V}(x,y)= c\,(V(x,y) +  \|(x,y)\|_{\mathcal{H}_2} ^2),
\end{equation} 
due to \eqref{app1} there exists some sufficiently large constant $c>0$ such that 
\[\Vert \hat{F}_1^\theta(t,x,y)\Vert^2\leq \hat{V}(x,y),\ \ \ \ \ (t,x,y) \in\,[0,T]\times \mathcal{H}_2.\]
Moreover,
\[\begin{array}{ll}
\ds{\Vert \hat{F}_1(t,x,y)-\hat{F}_1^\theta(t,x,y)\Vert}  &  \ds{\leq \Vert F_1(t,x,y)-F_1^\theta(t,x,y)\Vert +\delta\,\Vert x-x/(1+\theta \Vert x\Vert)\Vert}\\[10pt]
&\ds{\leq \theta V(x,y)+ \theta \frac{\delta\,\Vert x\Vert}{1+\theta\Vert x\Vert}\leq \theta\left(V(x,y)+\delta \Vert x\Vert\right)\leq \theta \hat{V}(x,y),}
\end{array}\]
if $c$ is sufficiently large.
As for  Hypothesis \ref{h3-app}, thanks to estimate \eqref{sy1000} below and  \eqref{eqh3} we get
\begin{equation}
\begin{aligned}
\int_0^T\mathbb{E}^n \hat{V}(X_{n}(t))\,dt \leq & \,c\int_0^T\left(\mathbb{E}^n V(X_n(t)) +  \mathbb{E}\|X_{n}(t)\| ^{2}\right)\,dt\leq c\left(V(x,y) +\|(x,y)\|_{\mathcal{H}_2} ^{2}\right).
\end{aligned}
\end{equation}
So we can recover \eqref{eqh3} for $\hat{V}$, once we take $c$ large enough.
	
 \end{enumerate}
  }	
\end{Remark}

For the nonlinear coefficient $F_2$, we make the following assumptions.
\begin{Hypothesis}\label{h3}
The mapping $F_2:\mathcal{H}_2 \to H$ is Lipschitz continuous. Moreover, if we define
\[
L_2:=\sup_{\substack{{x \in\,H}\\{y_1,\, y_2 \in\,H}}}\frac{\Vert F_2(x,y_1)- F_2(x,y_2)\Vert}{\Vert y_1-y_2\Vert},
\] 
we have  
\begin{equation}
\label{dissi}	
\omega:=\a_{2,1}- L_2 > 0.
\end{equation}

\end{Hypothesis}

Next, for every $\e>0$, we define the nonlinear mapping $F^{\e}:[0,T]\times D(F^\e)\to \mathcal{H}_2,$ by setting
\[
F^{\e}(t,x,y):=\left(F_1(t,x,y) ,\e^{-1} F_2(x,y)\right),\ \ \ \ t \in\,[0,T],\ \ (x,y) \in\,D(F^\e),
\]
with $D(F^\e):=D(F_1)\subseteq \mathcal{H}_2$.
By using  the notations we have introduced above, we can rewrite system \eqref{eq spde} as the following stochastic evolution equation on the Hilbert space $\mathcal{H}_2$
\begin{equation}\label{eq spde0}
dX(t) = \left[A^{\e} X(t) + F^{\e} (t,X(t))\right] dt + Q^{\e}dW^{}_t, \ \ \ \ X_\e(0)=(u_0, v_0),
\end{equation}
where $W_t=(W^1_t,W^2_t)$ is a cylindrical Wiener process in $\mathcal{H}_2$.

\medskip

In what follows, we shall  make the following fundamental  assumption.
\begin{Hypothesis}\label{h4}
For every $\e \in\,(0,1)$ and every initial condition $(u_0,v_0) \in\,D_V$,  there exists a martingale solution for \eqref{eq spde0}. This means that once fixed $(u_0,v_0) \in\,D_V$, for every $\e \in\,(0,1)$ there exist a stochastic base, $({\Omega}^{\e}, {\mathcal{F}}^{\e}, \{{\mathcal{F}}_t^{\e}\}_{t \geq 0}, {\mathbb{P}}^{\e})$, an adapted cylindrical Wiener process ${W}^{\e}_t=(W^{1,\e}_t, W^{2,\e}_t)$ in $\mathcal{H}_2$, and an adapted process ${X}^{\e} \in\,L^2(\Omega^\e;C([0,T];D(F^\e))$ such that
\[
{X}^{\e}(t) = S^{\e}(t)(u_0,v_0) + \int_0^t S^{\e}(t-s) F^{\e}(s, {X}^{\e}(s)) ds + {W}^{\e}_{A^{\e}} (t),\ \ \ \ t \in\,[0,T],
\]
where ${W}^{\e}_{A^\e}(t)$ is the   stochastic convolution
\[
{W}^{\e}_{A^\e}(t):= \int_0^t S^{\e}(t-s) Q^{\e} d{W}^{\e}_s,\ \ \ \ t\geq 0.
\] 
Moreover, there exists a constant $M>0$, independent of  $(u_0,v_0) \in\,D_V$,  such that 
\begin{equation}
\label{eq h4}
\sup_{\e \in\,(0,1)}\int_0^T\mathbb{E}^{\e} V(X^{\e}(t))\,dt\leq M \,V(u_0,v_0).
\end{equation}

\end{Hypothesis}

\begin{Remark}
	{\em In fact, the existence of a martingale solution for the slow-fast system \eqref{eq spde0}, for every initial condition $(u_0,v_0) \in\,D_V$ and every parameter $\e \in\,(0,1)$ is a consequence of Theorem }
\end{Remark}

In what follows, we shall denote
\[
{W}_{A_1}^{\e}(t)= \int_0^t S_1(t-s) Q_1 d{W}_s^{1,\e},\ \ \ \ {W}_{A_2}^{\e}(t)= \frac{1}{\sqrt{\e}}\int_0^t S_2\left((t-s)/\e\right) Q_2 d{W}_s^{2,\e}.
\]
With these notations, the slow and fast  components of $X^{\e}_{}(t)= (u^{\e}_{}(t), v^{\e}_{}(t))$, solve the equations
\begin{equation}\label{equ}
{u}^{\e}_{}(t) = S_1(t) u_0+ \int_0^t S_1(t-s)  F_{1}^{}(s, {X}^{\e}(s)) ds + W_{A_1}^{\e}(t),
\end{equation}
and 
\begin{equation}\label{eqv}
{v}^{\e}_{}(t) = S_2\left(t/\e\right) v_0 + \frac{1}{\e}\int_0^t S_2\left((t-s)/\e\right)  F_{2}( {X}^{\e}(s)) ds + W_{A_2}^{\e}(t).
\end{equation}

\section{Preliminaries and statement of the main result}

In this section, we will first review some properties of the asymptotic behavior of the fast motion with frozen slow component. Then we will introduce the coefficients of the candidate averaged equation and finally we will state the main result of the present paper. 

\subsection{The fast motion equation  with frozen slow variable}\label{fast motion}
We recall here some known facts about the ergodic properties of the fast motion equation with frozen slow variable. For all details we refer e.g. to \cite{SCMF1}.

For any fixed $x, y \in H$, we consider the following equation
\begin{equation}\label{eqfrozenslow}
\begin{aligned}
\begin{cases}
\ds{dv(t)= [A_2 v(t) + F_2(x, v(t))]\,dt + Q_2dW_t,}\\[10pt]
\ds{v(0)=y \in\,H,}
\end{cases}
\end{aligned}
\end{equation}
where $W_t$ is a cylindrical Wiener process on some stochastic basis $(\Omega, \mathcal{F}, \{\mathcal{F}_t\}_{t \in\,[0,T]}, \mathbb{P})$. Under Hypotheses \ref{h1} and \ref{h3}, equation \eqref{eqfrozenslow} admits a unique mild solution $v^{x, y}$ belonging to  $L^2(\Omega;C([0,T];H))$ such that for every $p\geq 1$
\[
\mathbb{E}\sup_{t \in [0,T]}\|v^{x,y}(t)\|^p < \infty.
\]
This allows us to introduce the transition semigroup $P^x_t$ associated with equation \eqref{eqfrozenslow}, which is defined by 
\[
P^x_t\varphi (y):= \mathbb{E}\,\varphi(v^{x,y}(t)), \ \ \ t\geq 0,\ \ \ y \in H,
\]
for any $\varphi \in B_b(H)$.

It is possible to show that for every  $0<\gamma <\gamma_2^\star$
\begin{equation}
\sup_{t \geq 0} \mathbb{E}\,\| v^{x,0}(t)\|_{D((-A_2)^{\gamma})}^{} \leq c\, (1+ \|x\|),
\end{equation}
(see also the proof of Lemma \ref{l1.3}). Hence, thanks to the Krylov-Bogoliubov theorem,  for every $x \in H$  the semigroup $P^x_t$ admits an invariant measure  $\mu^x$. 
Moreover, by using again arguments analogous to those we will use in  the proof of Lemma \ref{l1.3}, we can show that  for every $p \geq 1$,
\begin{equation}
\mathbb{E}\,\|v^{x, y}(t)\|^p \leq c_p\, \left(1 + \|x\|^p + e^{-\omega p t} \|y\|^p\right), \ \ \ \ \ \ \ t \geq 0.
\end{equation}
In particular, this implies that for every $p\geq 1$
\begin{equation}\label{eqinvariantmoment}
\int_H \|y\|^p \mu^x(dy) \leq c_p\,(1+ \|x\|^p).
\end{equation}

Now, we fix $x_1, x_2, y \in H$ and  we define $\Gamma_i(t):= v^{x_i,y}(t)- W_{A_2}(t)$, for $i=1,2$, where 
\[
W_{A_2}(t):= \int_0^t S_2(t-s) Q_2 dW_s.
\]
We have $\Gamma(t):=v^{x_1, y}(t)-v^{x_2,y}(t)= \Gamma_1(t) - \Gamma_2(t)$, and $\Gamma$ satisfies the following  equation
\[
\frac{d\Gamma(t)}{dt} = A_2 \Gamma(t) + F_2(x_1, v^{x_1, y}(t))- F_2(x_2, v^{x_2, y}(t)),\ \ \ \ \ \ \Gamma(0)=0.
\]
Thanks to Hypothesis \ref{h3}, we have
\[
\frac{1}{2}\frac{d}{dt}\|\Gamma(t)\|^2 \leq -\a_{2,0} \|\Gamma(t)\|^2 +\left([F_2]_{\text{\tiny{Lip}}}\|x_1-x_2\|+ L_2\,\|\Gamma(t)\|\right)\|\Gamma(t)\|,
\]
so that
\[
\frac{d}{dt}\|\Gamma(t)\|^2 \leq -\omega \|\Gamma(t)\|^2  + c\, \|x_1-x_2\|^2.
\]
By comparison, this gives
\[
\|\Gamma(t)\|^2 \leq c\int_0^t e^{-\omega(t-s)}\|x_1-x_2\|^2 ds \leq c\,\|x_1-x_2\|^2 ,
\]
which implies  
\begin{equation}\label{eq lip}
\mathbb{E}\sup_{t>0}\|v^{x_1, y}(t)-v^{x_2, y}(t)\|^2 \leq c\, \|x_1-x_2\|^2 .
\end{equation}
In a similar way, it is possible to prove that for every $x, y_1, y_2 \in\,H$ 
\[
\sup_{x\in H} \mathbb{E}\|v^{x, y_1}(t)-v^{x, y_2}(t)\|^2 \leq c\,e^{-\omega t} \|y_1-y_2\|^2,\ \ \ \ \ \ \ t\geq 0,
\]
for some  constant $c$ independent of $t$ and $y_1, y_2 \in\,H$, so that for every $\varphi \in\,\text{Lip}_b(H)$ we have
\begin{equation}
\label{sy13}
\left|P_t^x\varphi(y_1)-P_t^x\varphi(y_2)\right|\leq c\,[\varphi]_{\text{\tiny{Lip}}}e^{-\frac \omega 2\,t}\Vert y_1-y_2\Vert,\ \ \ \ \ t\geq 0.	
\end{equation}
This implies that $\mu^x$ is the unique invariant measure for $P^x_t$ and for every $x, y \in\,H$ and $\varphi \in\,\text{Lip}_b(H)$
\[\begin{array}{l}
\ds{\left|P^x_t\varphi(y)-\int_H \varphi(z)\,\mu^x(dz)\right|=\left|\int_H\left(P_t^x\varphi(y)-P_t^x\varphi(z)\right)\,\mu^x(dz)  \right|}\\[14pt]
\ds{\leq c\,[\varphi]_{\text{\tiny{Lip}}}e^{-\frac \omega 2\,t}\int_H\Vert y-z\Vert\,\mu^x(dz)\leq c\,[\varphi]_{\text{\tiny{Lip}}}e^{-\frac \omega 2\,t}\left(\Vert y\Vert+\int_H\Vert z\Vert\,\mu^x(dz)\right).}
	\end{array}
\]
Therefore, thanks to \eqref{eqinvariantmoment}, we obtain
\begin{equation}
\label{sy30}	
\left|P^x_t\varphi(y)-\int_H \varphi(z)\,\mu^x(dz)\right|\leq c\,[\varphi]_{\text{\tiny{Lip}}}e^{-\frac \omega 2\,t}\left(\Vert y\Vert+\Vert x\Vert+1\right).
\end{equation}

\subsection{The averaged nonlinear coefficient and its approximation}\label{sec G}

In what follows, in addition to Hypotheses \ref{h1} to \ref{h4}, we shall assume  the following condition.
\begin{Hypothesis}
	\label{h5}
	If $\mu^x(dy)$ is the invariant measure of the fast motion with frozen slow component $x$ introduced in Section \ref{fast motion}, then 
	\begin{equation}\label{sy22}
\bar{V}(x):=\int_{H} V(x, y)\,\mu^x(dy)<\infty,\ \ \ \ \ x \in \,\Pi_1 D_V,\end{equation}  
where
\[\Pi_1 D_V:=\left\{\,x \in\,H\,:\,(x,y) \in\,D_V,\ \text{{\em for some}}\ y \in\ H\,\right\}.\]
\end{Hypothesis}

In particular, Hypothesis \ref{h5} implies that for every $x \in\,\Pi_1 D_V$ the support of the invariant measure $\mu^x$ is contained in 
\[\Pi_2 D_V(x):=\left\{\,y \in\,H\,:\,(x,y) \in\,D_V\,\right\}.\]

Now, we define the averaged coefficient  $\bar{F}_1$ as\[\bar{F}^{}_1(t,x):= \int_H F^{}_1(t,x,y)\mu^x(dy),\ \ \ \ \ (t,x) \in\,[0,T]\times \Pi_1 D(F_1),\]
where
\[\Pi_1 D(F_1):=\left\{\,x \in\,H\, :\, (x,y) \in\,D(F_1)\ \text{for some}\ y \in\,H\,\right\}.\] 
Moreover, for any $\theta \in\,(0,1)$, we define the approximating averaged coefficient $\bar{F}^{\theta}_1$ as 
\[\bar{F}^{\theta}_1(t,x):= \int_H F^{\theta}_1(t,x,y)\mu^x(dy),\ \ \ \ \ (t,x) \in\,[0,T]\times H.\]
Next, for every $n \in\,\mathbb{N}$ and $(t,x,y) \in\,[0,T]\times \mathcal{H}_2$ we define
\[F^{\theta}_{1,n}(t,x,y):=\int_{\mathbb{R}^n}\rho_n(\xi-\Pi_n (x,y))F^\theta_1(t,R_n\,\xi)\,d\xi,\]
where $\{\rho_n\}_{n \in\,\mathbb{R}^n}$ is a sequence of non-negative smooth functions such that
\[\text{supp}\, \rho_n\subset \left\{\,\xi \in\,\mathbb{R}^x\,:\,|\xi|\leq 1/n\,\right\},
\ \ \ \ \int_{\mathbb{R}^n}\rho_n(\xi)\,d\xi=1,\]
and the mappings $\Pi_n:\mathcal{H}_2\to\mathbb{R}^n$ and $R_n:\mathbb{R}^n\to \mathcal{H}_2$ are defined by
\[\Pi_n(x,y)=\left (\langle (x,y),f_1\rangle_{\mathcal{H}_2},\cdots,\langle (x,y),f_n\rangle_{\mathcal{H}_2}\right),\ \ \ \ \ \ \ \ \ \ R_n\xi=\sum_{i=1}^n \xi_i f_i,\]
for some orthonormal basis $\{f_i\}_{i \in\,\mathbb{N}}\subset \mathcal{H}_2$.
Clearly, $F^\theta_{1,n}:[0,T]\times \mathcal{H}_2\to H$ is measurable and bounded, and $F^\theta_{1,n}(t,\cdot) \in\,C^\infty_b(\mathcal{H}_2;H)$, for every fixed $t \in\,[0,T]$, with
\begin{equation}
\label{sy60}
\sup_{t \in\,[0,T]}\Vert D^j_x 	F^\theta_{1,n}(t,\cdot)\Vert_0=:c_{j,n, \theta}<\infty,
\end{equation}
for every $j \in\,\mathbb{N}\cup \{0\}$, $n \in\,\mathbb{N}$ and $\theta \in\,(0,1)$. Moreover, for every $h \in\,H$
\begin{equation}
\label{sy50}
\lim_{n\to\infty}	\langle F^\theta_{1,n}(t,x,y),h\rangle=\langle F^\theta_1(t,x,y),h\rangle,\ \ \ (t,x,y)\in\,[0,T]\times \mathcal{H}_2,
\end{equation}
and 
\begin{equation}
\label{sy51}
\sup_{t \in\,[0,T]}\Vert F^\theta_{1,n}\Vert _0\leq \sup_{t \in\,[0,T]}\Vert F^\theta_{1}\Vert_0,\ \ \ \ \ n \in\,\mathbb{N}.	
\end{equation}

In what follows, we shall denote
\begin{equation}
\label{sy28}
\bar{F}^\theta_{1,n}(t,x):=\int_H 	F^\theta_{1,n}(t,x,y)\,\mu^x(dy),\ \ \ \ \ (t,x) \in\,[0,T]\times H.
\end{equation}

\begin{Lemma}
\label{lemma5.1}
Under Hypotheses \ref{h1}, \ref{h2} and \ref{h3},  the mapping $\bar{F}^\theta_{1, n}:[0,T]\times H\to H$ is measurable, for every $\theta \in\,(0,1)$ and $n \in\,\mathbb{N}$. Moreover, 	$\bar{F}^\theta_{1, n}(t,\cdot):H\to H$ is Lipschitz-continuous, uniformly with respect to $t \in\,[0,T]$, and 
\begin{equation}
\label{sy53}
\sup_{n \in\,\mathbb{N}}\sup_{t \in\,[0,T]}\Vert \bar{F}^\theta_{1, n}(t,\cdot)\Vert_0<\infty.	
\end{equation}

\end{Lemma}
\begin{proof}
According to \eqref{sy60}, the function ${F}^\theta_{1,n}(t,\cdot):\mathcal{H}_2\to H$ is  Lipschitz-continuous, uniformly with respect to $t \in\,[0,T]$. In particular, for every $h \in\,H$ the mapping
\[(x,y) \in\,\mathcal{H}_2\mapsto \langle {F}^\theta_{1,n}(t,x,y),h\rangle \in\,\mathbb{R},\]
is Lipschitz continuous, uniformly with respect to $t \in\,[0,T]$.
Hence, since
\[\langle \bar{F}^\theta_{1, n}(t,x),h\rangle=\int_H\langle F^\theta_{1, n}(t,x,y),h\rangle \,\mu^x(dy),\]
according to \eqref{sy30}, we have
\[\begin{array}{l}
\ds{\langle \bar{F}^\theta_{1,n}(t, x_1)-\bar{F}^\theta_{1, n}(t,x_2),h\rangle}\\[10pt]
\ds{ =\lim_{T\to\infty}\left(\mathbb{E}\,\langle {F}^\theta_{1,n}(t,x_1,v^{x_1,0}(T)),h\rangle-\mathbb{E}\,\langle {F}^\theta_{1, n}(t,x_2,v^{x_2,0}(T)),h\rangle \right).	}
\end{array}
\]
Now, in view of \eqref{eq lip} we have
\[\begin{array}{l}
\ds{\left |\mathbb{E}\,\langle {F}^\theta_{1,n}(t,x_1,v^{x_1,0}(T)),h\rangle-\mathbb{E}\,\langle {F}^\theta_{1,n}(t,x_2,v^{x_2,0}(T)),h\rangle\right|}\\[14pt]
\ds{\leq [{F}^\theta_{1,n}(t,\cdot)]_{\text{\tiny{Lip}}}\left(\Vert x_1-x_2\Vert+\mathbb{E}\,\Vert v^{x_1,0}(T)-v^{x_2,0}(T)\Vert \right)\Vert h\Vert	}\\[14pt]
\ds{\leq c\,[{F}^\theta_{1,n}(t,\cdot)]_{\text{\tiny{Lip}}}\Vert x_1-x_2\Vert\,\Vert h\Vert,}
\end{array}\]
and, due to the arbitrariness of $h \in\,H$, this implies
\begin{equation}
	\label{sy31}
\Vert \bar{F}^\theta_{1,n}(t, x_1)-\bar{F}^\theta_{1,n}(t,x_2)\Vert \leq c\,[{F}^\theta_{1,n}(t,\cdot)]_{\text{\tiny{Lip}}}\,\Vert x_1-x_2\Vert.	
\end{equation}
This means  that the mapping $x \in\,H\mapsto \bar{F}^\theta_{1,n}(t, x) \in\,H$ is Lipschitz-continuous, uniformly with respect to $t \in\,[0,T]$. 

\end{proof}

\begin{Lemma}
\label{lemma 5.2}
Under Hypotheses \ref{h1}, \ref{h2}, \ref{h3} and \ref{h5},  for every $\theta \in\,(0,1)$ the mapping 	$\bar{F}^{\theta}_1:[0,T]\times H\to H$ is measurable and bounded, and for every $ h \in\,H$ and $t \in\,[0,T]$ the mapping
\begin{equation}
\label{sy4000}	
x \in\,\mathcal{H}_2\mapsto \langle \bar{F}^\theta_1(t,x),h\rangle \in\,\mathbb{R},\end{equation}
is continuous.
Moreover, 
\begin{equation}
\label{sy26}
\Vert \bar{F}_1(t,x)-\bar{F}^{\theta}_1(t,x)\Vert^2 \leq \theta\,	\bar{V}(x),\ \ \ \ \ t \in\,[0,T],\ \ \ x \in\,\Pi_1 D(F_1).
\end{equation}
In particular, the mapping $\bar{F}_1:[0,T]\times \Pi_1 D(F_1)\to H$ is measurable.
\end{Lemma}
\begin{proof}
In view of \eqref{sy50} and \eqref{sy51}, from the dominated convergence theorem, for every $ h \in\,H$ we have
\[\begin{array}{ll}
\ds{\lim_{n\to\infty}\vert \langle \bar{F}^\theta_{1,n}(t,x),h\rangle	-\langle \bar{F}^\theta_1(t,x),h\rangle\vert=0,\ \ \ \ \ (t,x) \in\,[0,T]\times H.}
\end{array}
\]
Thanks to Lemma \ref{lemma5.1}, all $\bar{F}^\theta_{1,n}$ are measurable, and since $H$ is separable due to Petti's theorem the limit above implies that  the mapping $\bar{F}^\theta_1:[0,T]\times H\to H$ is measurable, for every $\theta \in\,(0,1)$. 

Now, thanks to condition 1. in Hypothesis \ref{h2}, for every $h \in\,H$ and $t \in\,[0,T]$ the mapping 
\[(x,y) \in\,\mathcal{H}_2\mapsto \langle F^\theta_1(t,x,y),h\rangle \in\,\mathbb{R},\] is continuous and bounded. Then, as proved in \cite{JLPL}, there is a sequence $\{\Phi^\theta_{h,n}(t,\cdot)\}_{n \in\,\mathbb{N}}$ in $\text{Lip}_b(\mathcal{H}_2)$ such that
\[\lim_{n\to\infty}\sup_{(x,y) \in\,\mathcal{H}_2}\vert \langle F^\theta_1(t,x,y),h\rangle-\Phi^\theta_{h,n}(t,x,y)\vert=0,\ \ \ \ t \in\,[0,T].\]
As shown in the proof of Lemma \ref{lemma5.1}, the function 
\[x \in\,H\mapsto \bar{\Phi}^\theta_{h,n}(t,x):=\int_H \Phi^\theta_{h,n}(t,x,y)\,\mu^x(dy) \in\,\mathbb{R},\]
is Lipschitz continuous. Therefore, since
\[\lim_{n\to\infty} \sup_{x \in\,H}\left|\int_H\left(\langle F^\theta_1(t,x,y),h\rangle-\Phi^\theta_{h,n}(t,x,y)\right)\,\mu^x(dy) \right|=0,\]
we can conclude that function \eqref{sy4000} is continuous.

Finally, due to \eqref{sy33} we have
\[\begin{array}{ll}
\ds{\Vert \bar{F}_1(t,x)-\bar{F}^{\theta}_1(t,x)\Vert^2\leq} & \ds{ \int_H \Vert {F}_1(t,x,y)-{F}^{\theta}_1(t,x,y)\Vert^2\,\mu^x(dy)}\\[10pt]
&\ds{\leq \theta\int_H V(x,y)\,\mu^x(dy)=\theta\,\bar{V}(x),\ \ \ \ \ t \in\,[0,T],\ \ \ x \in\,\Pi_1 D(F_1),}	
\end{array}\]
and \eqref{sy26} follows.
\end{proof}

\subsection{The main theorem}
Now, we can state the main result of this paper.

\begin{Theorem}\label{main}
Under Hypotheses \ref{h1} to \ref{h4}, for every initial condition $(u_0,v_0) \in\,\mathcal{H}_2$, with  $u_0\in D((-A_1)^{\beta})$, for some $\beta>0$, the family $\{\mathcal{L}(u^\e)\}_{\e \in\,(0,1)}$  is tight in $C([0,T];H)$. 

Moreover, if Hypothesis \ref{h5} holds and 
\begin{equation}
\label{sy20}
 \sup_{\e \in (0,1)}\int_0^T\mathbb{E}^\e\, \bar{V}(u^\e(t))\, dt <\infty,
\end{equation}
any weak limit $ \mu_{\text{\tiny{slow}}}$ of $\{\mathcal{L}(u^\e)\}_{\e \in\,(0,1)}$ in $C([0,T];H)$ solves the martingale problem with  data
$A_1$, $\bar{F}_1$, $Q_1$, and $u_0$, in the Hilbert space $H$.
Namely there exists a stochastic basis $(\bar{\Omega}, \bar{\mathcal{F}}, \{\bar{\mathcal{F}}_t\}_{t \geq 0}, \bar{\mathbb{P}})$, a cylindrical Wiener process $\bar{W}_t$ and an adapted process $\bar{u} \in L^2({\Omega}; C([0,T];D(\bar{F}_1)))$ such that $\mathcal{L}(\bar{u}) = \mu_{\text{\tiny{slow}}}$ and 
\begin{equation}
\label{sy7000}
\bar{u}(t) =S_1(t) u_0 + \int_0^t S_1(t-s) \bar{F}_1(s, \bar{u}(s)) ds + \int_0^t S_1(t-s) Q_1 d\bar{W}_s.
\end{equation}
\end{Theorem}

In the next two sections we will give a proof of Theorem \ref{main}: in Section \ref{sec3} we will prove tightness and in Section \ref{sec5} we will prove the validity of the averaging principle.

\section{Tightness}
\label{sec3}
We start with some a-priori bounds for the martingale solutions of equation \eqref{eq spde0}.

\subsection{Estimates for the stochastic convolution}

As a consequence of \eqref{reg}, it is possible to prove that under Hypothesis \ref{h1} for all $p\geq 1$ and $0\leq\gamma<\gamma_1$

\begin{equation}
\label{eq sc1}
\sup_{\e \in ( 0,1)}\,\mathbb{E}^{\e} \sup_{t \in [0,T]}\| {W}^{\e}_{A_1}(t)\|_{D((-A_1)^{\gamma})}^p <\infty.
\end{equation}
(for a proof see e.g. \cite[Section 5.4]{DP}).
Moreover, we have the following uniform bound.
\begin{Lemma} Under Hypothesis \ref{h1}, for every $p\geq 1$  we have
\begin{equation}\label{eq sc2}
\sup_{\e \in (0,1)}\,\sup_{\,t \geq 0}\,\mathbb{E}^{\e}\,\| {W}^{\e}_{A_2}(t)\|^p <\infty.
\end{equation}
\end{Lemma}

\begin{proof}Let $\{\beta^{2,\e}_k\}_{k \in \mathbb{N}}$ be a sequence of independent Brownian motions such that
\[
W^{2,\e}_t = \sum_{k=1}^{\infty} e_{2,k} \beta^{2,\e}_k(t),\ \ \ \ t\geq 0.
\]
Thanks to the Burkholder-Davis-Gundy inequality, we have
\[
\begin{aligned}
\mathbb{E}^\e\,\| {W}^{\e}_{A_2}(t)\|^p=&\e^{-p/2}\,\mathbb{E}^{\e}\left\| \int_0^t S_2((t-s)/\e)  \sum_{k=1}^{\infty} \lambda_{2,k}\,e_{2,k}\, d \beta^{2,\e}_k(s)\right\|^{p}\\[10pt]
\leq& \,c_p\,\e^{-p/2}\, \left(\,\sum_{k=1}^{\infty} \int_0^t e^{ -2\a_{2,k}\left( \frac{t-s}{\e}\right)}  \lambda_{2,k}^2 ds\right)^{p/2}\leq c_p \left(\,\sum_{k=1}^{\infty}  \a_{2,k}^{-1} \lambda_{2,k}^2\right)^{p/2}.\\
\end{aligned}
\]
Now, thanks to \eqref{alpha}, we have
\[
\sum_{k=1}^{\infty}  \a_{2,k}^{-1}\, \lambda_{2,k}^2  < \a_{2,1}^{-2\gamma_2}\sum_{k=1}^{\infty}  \a_{2,k}^{2\gamma_2-1} \lambda_{2,k}^2< +\infty,
\]
and  \eqref{eq sc2} follows.
\end{proof}

Notice that we also have tha for every $p\geq 1$
\begin{equation}
\label{eq sc2-bis}
\sup_{\e \in\,(0,1)}\mathbb{E}^\e \sup_{t \in\,[0,T]}\Vert W^\e_{A_2}\Vert^p<\infty.	
\end{equation}

Next, we want to investigate the time-continuity  of the stochastic convolution. Thanks to Hypothesis \ref{h1}, 
for every $p\geq 1$ and $0\leq \gamma < \gamma_1^\star:=\gamma_1\wedge 1/2$ we have
\begin{equation}
\label{eq scc1}
\sup_{\e \in\,(0,1)}{\mathbb{E}}^{\e}\,\|{W}_{A_1}^{\e}(t+h) - {W}_{A_1}^{\e}(t)\|_{D((-A_1)^{\gamma})}^p \leq c\, h^{(\gamma_1^\star-\gamma)\,p},\ \ \ \ \ \ t \geq 0,\ \ h \in\,(0,1).
\end{equation}

The Garcia-Rademich-Rumsey theorem, together with \eqref{eq scc1}, imply that there exists $\beta>0$ such that
\begin{equation}
\label{eq scc1.1}
\begin{array}{l}
\ds{\sup_{\e \in\,(0,1)}\mathbb{E}^{\e} \sup_{t,s, \in [0,T]} \|W_{A_1}^{\e}(t)-W_{A_1}^{\e}(s)\|_{D((-A_1)^{\gamma})}^2\,|t-s|^{-\beta}<\infty.}
\end{array}
\end{equation}

\begin{Lemma}\label{l0.2}
Assume Hypothesis \ref{h1} hold and fix 
 $0\leq \gamma < \gamma_2^\star:=\gamma_2\wedge 1/2$. Then,  we have
\begin{equation}\label{eq scc2}
\sup_{\e \in\,(0,1)}\e^{(\gamma_2^\star-\gamma)2}\,\mathbb{E}^{\e}\,\|{W}_{A_2}^{\e}(t+h) - {W}_{A_2}^{\e}(t)\|_{D((-A_2)^{\gamma})}^2 \leq c\, h^{(\gamma_2^\star -\gamma)2},\ \ \ \ t \geq 0,\ \ h \in\,(0,1).
\end{equation}
\end{Lemma}
\begin{proof}
We recall that
\[
{W}_{A_2}^{\e}(t)= \frac{1}{\sqrt{\e}}\int_0^t S_2\left((t-s)/\e\right) Q_2\, d{W}_s^{2,\e}.
\]
Therefore, since
\[
\mathcal{L}\left(\frac{1}{\sqrt{\e}}\,W^{2,\e}_{A_2}(\cdot)\right) = \mathcal{L}\left(W^{2,\e}_{A_2}(\cdot/\e)\right),
\]
we have
\[
\mathcal{L}\left( {W}_{A_2}^{\e}(\cdot) \right)= \mathcal{L}\left( \int_0^{\cdot/\e} S_2\left(\cdot/\e-s\right) Q_2\, d{W}_s^{2,\e}  \right).
\]
In particular, since $A_2$ satisfies the same assumptions as $A_1$, we can apply \eqref{eq scc1}, and we get \eqref{eq scc2}.
\end{proof}

\subsection{Estimates for the slow motion $u^{\e}(t)$}\label{sec u}

 For every $\e \in\,(0,1)$, we define
\[
\Psi^{\e} (t):= \int_0^t S_1(t-s) F_{1}(s, {X}^{\e}(s)) ds,\ \ \ \ \ \ \ t \in\,[0,T].
\]
By using the same arguments used in the proof of Lemma \ref{l0.3} and Lemma \ref{l0.4}, it is possible to prove that  under  Hypotheses \ref{h1}, \ref{h2} and \ref{h4}, for every $\gamma \in [0,1/2)$ 
\begin{equation}
\label{sy1}
\sup_{\e \in (0,1)}
\mathbb{E}^{\e}\, \sup_{t \in [0,T]} \|\Psi^{\e}(t)\|_{D((-A_1)^\gamma)}^2\leq c_{T,\gamma}V(u_0,v_0),\ \ \ \ t \in\,[0,T].
\end{equation}
Moreover, there exists some $\beta>0$ such that
\begin{equation}
\label{sy2}
\sup_{\e \in (0,1)}\mathbb{E}^{\e} \sup_{t,s \in [0,T]}\|\Psi^{\e}(t)- \Psi^{\e}(s)\|_{D((-A_1)^{\gamma})}^{2}|t-s|^{-\beta} \leq c_T\,V(u_0,v_0).
\end{equation}

Therefore, it is possible to prove the following result.

\begin{Lemma}\label{l1.1}
Assume that Hypotheses \ref{h1}, \ref{h2} and \ref{h4} hold. Then for any initial condition $ (u_0,v_0)\in\,D_V$ and $0\leq \gamma < \gamma_1^\star$, we have
\begin{equation}\label{l1.1eq1}
\sup_{ \e \in (0,1)}{\mathbb{E}}^{\e}\|{u}^\e(t)\|_{D((-A_1)^{\gamma})}^{2} \leq c\left(t^{-2\gamma }+1\right)\,\left(1+\Vert u_0\Vert^2+V(u_0,v_0)\right),\ \ \ \ \ t \in\,[0,T].
\end{equation}
Moreover,
\begin{equation}\label{l1.1eq2}
\begin{aligned}
\sup_{ \e \in (0,1)} {\mathbb{E}}^{\e}\,\|{u}^\e(t)-{u}^{\e}(s)\|_{D((-A_1)^{\gamma})}^{2}
 \leq c\, \rho_{\gamma}(s,t)\left(1+\Vert u_0\Vert^2+V(u_0,v_0)\right),\ \ \ \ \ \ \ 0\leq s\leq t\leq T,
\end{aligned}
\end{equation}
where 
\begin{equation}
\label{sy6}
\rho_{\gamma}(s,t):= \left(\int_s^{t} r^{-(\gamma +1)} dr \right)^{2} + (t-s)^{\beta} + (t-s)^{2(\gamma_1^\star- \gamma)},\ \ \ \ \ \ \ s\leq t,
\end{equation}
and $\beta$ is the  constant introduced in \eqref{sy2}. Finally, if $u_0 \in D((-A_1)^{\delta})$, for some $\delta >\gamma$, then there exists  $\eta=\eta(\gamma,\delta)>0$ such that 
\begin{equation}\label{l1.1eq3}
\sup_{ \e \in (0,1)} {\mathbb{E}}^{\e}\,\|{u}^\e(t)-{u}^{\e}(s)\|_{D((-A_1)^{\gamma})}^{2}
 \leq c\, (t-s)^\eta\left(1+\Vert u_0\Vert^2_{D((-A_1)^\delta)}+V(u_0,v_0)\right),\ \ \ \ \ \ \ 0\leq s\leq t\leq T.\end{equation}

\end{Lemma}

\begin{proof}
Since 
\[\Vert S_1(t) u_0\Vert_{D((-A_1)^{\gamma}}\leq c\,t^{-\gamma}\Vert u_0\Vert,\]
\eqref{l1.1eq1} follows from \eqref{eq sc1} and \eqref{sy1}. Moreover, 
for every $\d\ \in\,[0,\gamma+1]$, we have
\[
\begin{aligned}
\|S_1(t)u_0 - S_1(s)u_0 \|_{D((-A_1)^{\gamma})}&\leq  \int_s^t \left\|(-A_1)^{\gamma+1-\delta}S_1(r) (-A_1)^{\delta}u_0\right\|\,dr\\[10pt]
&\leq \int_s^t r^{\delta-(\gamma+1)} dr\ \| u_0 \|_{D((-A_1)^{\delta})}.
\end{aligned}
\]
Therefore, thanks to \eqref{eq scc1}, and \eqref{sy2}, if we take $\delta=0$
we get  \eqref{l1.1eq2}.

Now, since
\[
a^\la-b^\la \leq (a-b)^{\la},\ \ \ \ \ \la \in (0,1),\ \ \ \ 0\leq b\leq a, 
\]
if we take $\delta \in\,(\gamma, \gamma+1]$ we have
\[
\int_s^t r^{\delta-(\gamma+1)}\,dr=\frac 1{\delta-\gamma}\left(t^{\delta-\gamma}- s^{\delta- \gamma}\right) \leq (t-s)^{\delta- \gamma},\ \ \ \ \ t>s.
\]
Therefore, \eqref{l1.1eq3} holds for any $\gamma<\gamma_1^\star$ and $\delta>\gamma$.
\end{proof}

Notice that from \eqref{eq sc1} and \eqref{sy1} we have also
\begin{equation}
\label{sy1000}
\sup_{\e \in\,(0,1)} \mathbb{E}^\e \sup_{t \in\,[0,T]} \Vert u^\e(t)\Vert\leq c_T\,\left(1+\Vert u_0\Vert^2+V(u_0,v_0)\right).	
\end{equation}

\subsection{Estimates for the fast motion $v^{\e}(t)$}
\begin{Lemma}\label{l1.3}
Under Hypotheses \ref{h1} to \ref{h4} we have
\begin{equation}
\sup_{\e \in (0,1)}{\mathbb{E}}^{\e}\sup_{t \in\,[0,T] }\|{v}^{\e}(t)\|^{2} <\infty.
\end{equation}
\end{Lemma}
\begin{proof}
If we define $\Gamma^{\e}(t):= {v}^{\e}_{}(t)- {W}_{A_2}^{\e}(t)$, we have
\[
\begin{aligned}
\frac{d\Gamma^{\e}(t)}{dt} =& \frac 1\e\, A_2\Gamma^{\e}(t) + \frac 1\e (F_{2}({u}^\e(t), \Gamma^{\e}(t) + {W}_{A_2}^{\e}(t)) - F_{2}({u}^\e(t), {W}_{A_2}^{\e}(t)) )\\[14pt]
&+ \frac 1\e\,F_{2}({u}^\e(t),  {W}_{A_2}^{\e}(t)).\\
\end{aligned}
\]
Thanks to Hypothesis \ref{h3}, we have
\[
\frac{1}{2} \frac{d}{dt} \|\Gamma^{\e}(t)\|^{2} \leq - \frac{\omega}{\e}  \|\Gamma^{\e}(t)\|^{2} + \frac{1}{\e} \|F_2(u^{\e}(t), {W}^{\e}_{A_2}(t))\| \|\Gamma^{\e}(t)\|^{},
\]
and the Young's inequality gives
\[
\begin{aligned}
& \frac{d}{dt} \|\Gamma^{\e}(t)\|^{2} \leq -\frac{\omega}{2\e} \|\Gamma^{\e}(t)\|^{2}+  \frac{c}{\e} (1 + \|{u}^{\e}(t)\|^{2} + \| {W}^{\e}_{A_2}(t) \|^2).
\end{aligned}
\]
By comparison, this implies
\[
\begin{array}{ll}
\ds{\|{v}^\e(t)\|^{2} }&\ds{ \leq\,2\,\|\Gamma^{\e}(t)\|^{2} +2\, \|{W}_{A_2}^{\e}(t)\|^{2}}\\[10pt]
& \ds{\leq  2\, e^{-\frac{t\omega}{\e}} \|v_0\|^{2}+\frac{c}{\e} \int_0^t e^{-\frac{(t-s)\omega}{2\e}} (1 + \|{u}^{\e}(s)\|^{2} + \|{W}_{A_2}^{\e}(s)\|^{2}) ds+ c\, \|{W}_{A_2}^{\e}(t)\|^{2}.}
\end{array}
\]
Therefore, in view of \eqref{eq sc2-bis} and \eqref{sy1000}, this gives
\[
\begin{aligned}
{\mathbb{E}}^{\e}\sup_{t \in\,[0,T]}\|{v}^\e(t)\|^{2} \leq  c \left(1+\Vert (u_0,v_0)\Vert_{\mathcal{H}_2}^2+V(u_0,v_0)\right),
\end{aligned}
\]
where $c $ is a constant independent of $\e \in\,(0,1)$.
\end{proof}

\subsection{Proof of the tightness}
If we define $y^{\e}(t) := u^\e(t)-S_1(t)u_0$, thanks to \eqref{eq sc1} and \eqref{sy1},  we have that for  $\gamma<\gamma_1^\star$ 
\begin{equation}  \label{sy1001}
\sup_{\e \in (0,1)}\mathbb{E}^{\e}\ \sup_{t \in [0,T]} \|y^{\e}(t)\|_{D((-A_1)^\gamma)}^{2}<\infty.
\end{equation}
Moreover, thanks to \eqref{eq scc1.1} and \eqref{sy2}, there exists  some  $\beta>0$ sufficiently small, we have
\begin{equation}  \label{sy1002}
\sup_{\e \in (0,1)}\mathbb{E}^{\e} \sup_{t,s \in [0,T]}\|y^{\e}(t)-y^{\e}(s)\|^{2} ds\ |t-s|^{-\beta}<\infty.
\end{equation}

\begin{Lemma}\label{l4.1.6}
Under Hypotheses \ref{h1} to \ref{h4}, the family of measures $\{ \mathcal{L}(u^\e)\}_{\e \in (0,1)}$ is tight in $C([0,T];H)$. 
\end{Lemma}
\begin{proof}
Let $\gamma$ and $\beta$ be the same constants as in \eqref{sy1001} and \eqref{sy1002}. If we fix an arbitrary $R>0$ and introduce the compact set of $C([0,T];H)$ 
\[
K_R:= \left\{ u \in C([0,T];H)\, :\, \Vert u \Vert_{C^{\gamma}([0,T];H)}+\sup_{t,s\in [0,T]} \frac{\|u(t)- u(s)\Vert^2}{|t-s|^{\beta}} \leq R \right\},
\]
thanks to \eqref{sy1001} and \eqref{sy1002} we have
\[
\sup_{\e \in (0,1)}\mathbb{P}^{\e}\left(y^{\e}\in K_R^c\right) \leq \frac{c_T}R.
\]
Therefore, for every $\eta>0$ we can fix $R_\eta>0$ such 
\[\sup_{\e \in (0,1)}\mathbb{P}^{\e}\left(u^{\e}\in S_1(\cdot)u_0+K_{R_\eta}^c\right)=
\sup_{\e \in (0,1)}\mathbb{P}^{\e}\left(y^{\e}\in K_{R_\eta}^c\right) \leq \eta.
\]
Since $K_{R_\eta}+S_1(\cdot)u_0$ is also compact in $C([0,T];H)$, due to the arbitrariness of $\eta>0$ we conclude that
the family of measures $\{ \mathcal{L}(u^{\e})\}_{\e \in (0,1)}$ is tight on $C([0,T];H)$. 
\end{proof}

\section{ The averaging limit}
\label{sec5}
In this section we conclude the proof of Theorem \ref{main} by proving that any weak limit of $\mathcal{L}(u^\e)$ is a martingale solution of problem \eqref{sy7000}.

\subsection{A time-discretization of the fast motion}\label{sec aux}
For every $\d>0$ and $k= 1,2,...,[T/\d]$, we have that ${v}^{\e}$  satisfies the following equation
\[
\begin{aligned}
{v}^{\e}(t)=&\ S_2\left((t-k\d)/\e\right) {v}^{\e}(k\d) + \frac{1}{\e}\int_{k\d}^t S_2\left((t-s)/\e\right)  F_2({u}^{\e}(s),{v}^{\e}(s))  ds\\[10pt]
&\ \ \ \ \ \ \ \ \ \ \ \   + \frac{1}{\sqrt{\e}} \int_{k \d}^{t}  S_2\left((t-s)/\e\right) Q_2 d {W}^{2,\e}_s,\ \ \ t\in [k\d,(k+1)\d).
\end{aligned}
\]
Now, we denote by ${v}^{\e,\d}$ the solution to the following equation
\begin{equation}\label{eq4.4}
\begin{aligned}
{v}^{\e,\d}(t)=& S_2\left((t-k\d)/\e\right) {v}^{\e}(k\d) + \frac{1}{\e}\int_{k\d}^t S_2\left((t-s)/\e\right)  F_2({u}^{\e}(k\d),{v}^{\e,\d}(s))  ds\\[10pt]
&\ \ \ \ \ \ \ \ \ \ \ \   + \frac{1}{\sqrt{\e}} \int_{k \d}^{t}S_2\left((t-s)/\e\right)  Q_2 d{W}^{2,\e}_s,\ \ \ \ \ \ \ \ \ t\in [k\d,(k+1)\d).
\end{aligned}
\end{equation}
Moreover, we denote by ${u}^{\e,\d}$ the process defined by
\begin{equation}
{u}^{\e,\d}(t) = \sum_{k =0}^{[T/\d]} {u}^{\e}(k\d) \mathbbm{1}_{\{k\d\leq t< (k+1)\d\}},\ \ \ \ t \in [0,T].
\end{equation}
\begin{Lemma}\label{l4.2}
Assume Hypotheses \ref{h1} to \ref{h4} and fix $(u_0,v_0) \in D_V$. Then, if $\rho_0$ is the function defined in \eqref{sy6} (when $\gamma=0$), for any $k \leq [T/\d]$ and $t \in [k\d,(k+1)\d)$ we have
\begin{equation}
\label{sy15}
\sup_{\e \in (0,1)}{\mathbb{E}^{\e}}\,\|{u}^{\e}(t)- {u}^{\e,\d}(t)\|^{2}\leq c\,\rho_{0}(t,k\d) \left(1+\Vert u_0\Vert^2 +V(u_0,v_0)\right),
\end{equation}
and
\begin{equation}\label{eq4.8}
\begin{aligned}
{\mathbb{E}^{\e}}\,\|{v}^{\e,\d}(t)- {v}^{\e}(t)\|^2\leq \frac {c}{\e}\left(1+\Vert u_0\Vert^2 +V(u_0,v_0)\right)\exp\left(\frac{c\,\d}\e\right)  \int_{k\d}^t \rho_0(s,k\d) ds,\ \ \ \ \e \in (0,1).
\end{aligned}
\end{equation}
\end{Lemma}
\begin{proof}
In order to prove \eqref{sy15}, we notice that as a consequence of Lemma \ref{l1.1} 
\begin{equation}\label{eq4.6}
\begin{aligned}
{\mathbb{E}}^{\e}\,\|{u}^{\e}(t)- {u}^{\e,\d}(t)\|^{2} = {\mathbb{E}}^{\e}\,\|{u}^{\e}(t)- {u}^{\e}(k\d)\|^{2} \leq c  \rho_{0}(t,k\d)\left(1+\Vert u_0\Vert^2 +V(u_0,v_0)\right).
\end{aligned}
\end{equation}
In order to prove \eqref{eq4.8}, we notice that for $t \in [k\d, (k+1)\d)$ we have 
\begin{equation}
\begin{array}{l}
\ds{{\mathbb{E}}^{\e}\,\|{v}^{\e,\d}(t)- {v}^{\e}(t)\|^2}\\[14pt]
\ds{\leq  \frac{1}{\e^2} {\mathbb{E}}^{\e}\left(\int_{k\d}^{t}\left\|S_2\left((t-s)/\e\right) \left( F_2({u}^{\e}(k\d),v^{\e,\d}(s))-  F_2({u}^{\e}(s),{v}^{\e}(s))\right) \right\|ds\right)^2}\\[18pt]
\ds{\leq  \frac{c}{\e^2}\int_{k\d}^{t}e^{-\frac{2\a_{2,0}}{\e}(t-s)}ds \left(\int_{k\d}^t  {\mathbb{E}}^{\e}\,\|{u}^{\e}(s)- {u}^{\e}(k\d)\|^{2} ds +  \int_{k\d}^t  {\mathbb{E}}^{\e}\,\|{v}^{\e,\d}(s)- {v}^{\e}(s)\|^{2}ds\right).}
\end{array}
\end{equation}
Hence, thanks to \eqref{eq4.6}, we get
\begin{equation}
\begin{aligned}
{\mathbb{E}}^{\e}\,\|{v}^{\e,\d}(t)- {v}^{\e}(t)\|^2\leq& \,\frac{c}{\e}   \int_{k\d}^t \rho_0(s,k\d) ds + \frac c{\e}  \int_{k\d}^t  {\mathbb{E}}^{\e}\,\|{v}^{\e,\d}(s)- {v}^{\e}(s)\|^{2} ds,
\end{aligned}
\end{equation}
and by Gronwall's inequality, this implies
\begin{equation}\label{eq 4.6.1}
\begin{aligned}
{\mathbb{E}}^{\e}\,\|{v}^{\e,\d}(t)- {v}^{\e}(t)\|^2 \leq \frac {c}{\e}\left(1+\Vert u_0\Vert^2 +V(u_0,v_0)\right)\exp\left(\frac{c\,\d}\e\right) \int_{k\d}^t \rho_0(s,k\d) ds.
\end{aligned}
\end{equation}
\end{proof}

\subsection{Some preliminary results}
\label{ssec5.2}

For any $\xi \in\ C([0,T];H)$, we define 
\[\begin{aligned}
G_\xi(t,x,y): = \langle  F_1(t,x,y),  \xi(t)\rangle,\ \ \ \ t \in\,[0,T],\ \ \ \ (x,y) \in\,D(F_1),
\end{aligned}
\]
and for every $\theta \in\,(0,1)$ we define
\[
\begin{aligned}
G_{\xi}^{\theta}(t,x,y): = \langle  F^{\theta}_1(t,x,y),   \xi(t)\rangle\ \ \ \ t \in\,[0,T],\ \ \ \ (x,y) \in\,\mathcal{H}_2.
\end{aligned}
\]
The mapping $G_\xi:[0,T]\times \mathcal{H}_2\to H$ is measurable, and, due to Condition 1. in Hypothesis \ref{h2}, it is easy to check that   the mapping $G^\theta_\xi:[0,T]\times \mathcal{H}_2\to H$ is continuous and bounded. Moreover,
as a consequence of \eqref{sy33}, we have 
\begin{equation}  \label{sy73}
|G_\xi(t,x,y)- G_{\xi}^\theta(t,x,y)| \leq \theta \sup_{t \in\,[0.T]}\| \xi(t)\|\, V(x,y),\ \ \ \ t \in\,[0,T],\ \ \ \ (x,y) \in\,D(F_1).
\end{equation}

Next we define  $\bar{G}_\xi$ and $\bar{G}_{\xi}^\theta$ by setting
\[
\bar{G}_\xi(t,x): =  \langle \bar{F}_1(t,x),   \xi(t)\rangle,\ \ \ \ \ (t,x) \in\,[0,T]\times \Pi_1 D(F_1),
\]
and
\[
\bar{G}_{\xi}^{\theta}(t,x): = \langle \bar{F}^{\theta}_1(t,x),   \xi(t)\rangle,\ \ \ \ \ (t,x) \in\,[0,T]\times H.
\]
Thanks to \eqref{sy26}, we have
\begin{equation}\label{eqG1}
|\bar{G}_\xi(t,x)-\bar{G}_{\xi}^{\theta}(t,x)|^2 \leq \theta\, \sup_{t \in\,[0,T]}\| \xi(t)\|^2\,\bar{V}(x),\ \ \ \ \ (t,x) \in\,[0,T]\times \Pi_1D(F_1).
\end{equation}
According to Lemma \ref{lemma 5.2}, we have that the mapping $\bar{G}_\xi:[0,T]\times D\to \mathbb{R}$ is measurable and for every $\theta \in\,(0,1)$ the mapping $\bar{G}^\theta_\xi:[0,T]\times H\to\mathbb{R}$  is continuous.

As we already mentioned above, in \cite{JLPL} it is shown that the space $\text{Lip}_b([0,T]\times H)$ is dense in 
$C_b([0,T]\times H)$. Then, we can fix a sequence $\{G^\theta_{\xi, n}\}_{n \in\,\mathbb{N}}\subset \text{Lip}_b([0,T]\times H)$ such that
\begin{equation}
\label{sy27-bis}
\lim_{n\to\infty}\,\sup_{(t,x,y) \in\,[0,T]\times \mathcal{H}_2}\vert G^\theta_{\xi,n}(t,x,y)-G^\theta_\xi(t,x,y) \vert =0.
\end{equation}
Furthermore, if we define
\[
\bar{G}^\theta_{\xi ,n}(t,x): = \int_H G^\theta_{\xi, n}(t,x,y) \mu^x(dy),\ \ \ \ \ t \in\,[0,T],\ \ \ x\in\,H, 
\]
by using arguments analogous to those used in Lemma \ref{lemma5.1} for $\bar{F}^\theta_n$, we can show that each mapping 
$\bar{G}^\theta_{\xi,n}:[0,T] \times H \to \mathbb{R}$ is  Lipschitz-continuous. Moreover, 
\begin{equation}
\label{sy34-bis}
\lim_{n\to\infty} \sup_{(t,x) \in\,[0,T]\times H}\vert \bar{G}^\theta_{\xi,n}(t,x)	-\bar{G}^\theta_\xi(t,x)\vert=0.
\end{equation}

\begin{Lemma}\label{l aux2}
Under Hypotheses \ref{h1} to \ref{h5},  we have
\begin{equation}\label{eq aux2.1}
\lim_{\e\to 0}\, {\mathbb{E}}^\e\sup_{t \in [0,T]}\left|\int_0^{t} \langle  F^1(s, u^\e(s), v^\e(s)) , \xi(s) \rangle ds- \int_0^{t} \langle  \bar{F}^1(s, {u}^\e(s)) , \xi(s) \rangle ds \right|=0.
\end{equation}
\end{Lemma}
\begin{proof}
By using the notations we have just introduced, \eqref{eq aux2.1} can be rewritten as
\begin{equation}\label{eq4.2}
\lim_{m \to \infty} {\mathbb{E}}^{\e} \sup_{t \in [0,T]} \left| \int_0^t \left[G_\xi(s, {u}^{\e}(s), {v}^{\e}(s))- \bar{G}_\xi(s, {u}^{\e}(s))\right] ds\right|=0.
\end{equation}
For every $\theta \in\,(0,1)$ and $n \in\,\mathbb{N}$, we have
\begin{equation}\label{eq4.3}
\begin{aligned}
 & {\mathbb{E}}^{\e}\sup_{t \in [0,T]}\left| \int_0^t \left[G_{\xi}(s, {u}^{\e}(s), {v}^{\e}(s)) - \bar{G}_{\xi}(s, {u}^{\e}(s))\right]ds\right|\\[10pt]
& \leq  {\mathbb{E}}^\e \sup_{t \in [0,T]}\left|\int_0^t \left[G_\xi(s, {u}^{\e}(s), {v}^{\e}(s))-G^\theta_{\xi,n}(s, {u}^{\e}(s), {v}^{\e}(s))\right] ds\right| \\[10pt]
& \ \ \ \ \ \ \ \ \ \ \ +  {\mathbb{E}}^{\e}\sup_{t \in [0,T]}\left| \int_0^t \left[G^{\theta}_{\xi,n}(s, {u}^{\e}(s), {v}^{\e}(s))-\bar{G}^\theta_{\xi,n}(s, {{u}}^{\e}(s))\right]ds\right| \\[10pt]
& \ \ \ \ \ \ \ \ \ \ \ +  {\mathbb{E}}^\e \sup_{t \in [0,T]}\left| \int_0^t \left[\bar{G}^\theta_{\xi,n}(s, {u}^{\e}(s)) - \bar{G}_\xi(s, {u}^{\e}(s))\right] ds\right|=:\sum_{k=1}^3 I_{\e, k}(\theta, n).
\end{aligned}
\end{equation}

For the term $I_{\e, 1}(\theta, n)$, due to \eqref{sy73}, we have
\begin{equation}
\begin{array}{l}
\ds{I_{\e, 1}(\theta, n) \leq \int_0^T  {\mathbb{E}}^{\e} \left|G_\xi(s, {u}^{\e}(s), {v}^{\e}(s))-G_{\xi}^{\theta}(s, {u}^{\e}(s), {v}^{\e}(s))\right|ds}\\[14pt]
 \ds{+\int_0^T  {\mathbb{E}}^{\e} \left|G_{\xi}^{\theta}(s, {u}^{\e}(s), {v}^{\e}(s))-G^\theta_{\xi,n}(s, {u}^{\e}(s), {v}^{\e}(s))\right| ds}\\[14pt]
\ds{\leq\  \theta \int_0^T {\mathbb{E}}^{\e} \,V({u}^{\e}(s), {v}^{\e}(s))\, ds+\int_0^T {\mathbb{E}}^\e \left| G_{\xi}^{\theta}(s, {u}^{\e}(s), {v}^{\e}(s))-G^\theta_{\xi,n}(s, {u}^{\e}(s), {v}^{\e}(s))\right|ds.}
\end{array}
\end{equation}
In view of Hypothesis \ref{h4}, this gives
\begin{equation}
	\label{sy80}
\begin{array}{ll}
\ds{	I_{\e, 1}(\theta, n) \leq } & \ds{c_T(u_0,v_0)\,\theta+c_T\,\sup_{(t,x,y) \in\,[0,T]\times \mathcal{H}_2}\vert G^\theta_{\xi}(t,x,y)-G^\theta_{\xi,n}(t,x,y)\vert_.}
\end{array}
	\end{equation}

For the term $I_{\e, 3}(\theta, n)$, according to \eqref{eqG1} we have
\[
\begin{aligned}
I_{\e, 3}(\theta, n) \leq& \int_0^T {\mathbb{E}}^\e\left| \bar{G}_{\xi}(s, {u}^{\e}(s))-\bar{G}_{\xi}^\theta(s, {u}^{\e}(s))\right|ds+ \int_0^T {\mathbb{E}}^\e\left|\bar{G}_{\xi}^\theta(s, {u}^{\e}(s))-\bar{G}^\theta_{\xi, n}(s, {u}^{\e}(s))\right| ds \\[14pt]
\leq&\ \theta \int_0^T {\mathbb{E}}^\e\, \bar{V}({u}^{\e}(s))\,  ds+\int_0^T {\mathbb{E}}^\e\left| \bar{G}_{\xi}^{\theta}(s, {u}^{\e}(s))-\bar{G}^\theta_{\xi,n}(s, {u}^{\e}(s))\right|\,ds. \\
\end{aligned}
\]
Thanks to Hypothesis \ref{h5} and \eqref{sy20}, this implies
\[
\begin{aligned}
I_{\e, 3}(\theta, n)\leq c\, \theta +T\,\sup_{(t,x) \in\,[0,T]\times H}\vert\bar{G}_{\xi}^{\theta}(t,x)-\bar{G}^\theta_{\xi,n}(t,x)\vert.
\end{aligned}
\]
Due to \eqref{sy27-bis} and \eqref{sy34-bis}, the inequality above, together with \eqref{sy80}, implies that for every $\eta>0$ there exist $\theta_\eta \in\,(0,1)$ and $n_\eta \in\,\mathbb{N}$ such that
\begin{equation}
\label{sy81}
\sup_{\e \in\,(0,1)}\,\left(I_{\e, 1} (\theta_\eta, n_\eta)+I_{\e, 3} (\theta_\eta, n_\eta)\right)<\eta.	
\end{equation}

Finally, concerning $I_{\e, 2}(\theta_\eta,n_\eta)$,  if $u^{\e,\d}$ and $v^{\e,\d}$ are the auxiliary process introduced in Section \ref{sec aux}, we have
\begin{equation}
\begin{array}{ll}
\ds{I_{\e, 2}(\theta_\eta,n_\eta)\leq} & \ds{ \int_0^T {\mathbb{E}}^\e \left|G^{\theta_\eta}_{\xi, n_\eta}(s, {u}^{\e}(s), {v}^{\e}(s)) -G^{\theta_\eta}_{\xi,n_\eta}(s, {u}^{\e,\d}(t), {v}^{\e,\d}(t))\right|\,ds}\\[14pt]
&\ds{+\int_0^T {\mathbb{E}}^\e\left|\bar{G}^{\theta_\eta}_{\xi, n_\eta}(s, {u}^{\e}(s))-   \bar{G}^{\theta_\eta}_{\xi,n_\eta} (t, {u}^{\e,\d}(s)) \right| ds}\\[14pt]
&\ds{+ \int_0^T {\mathbb{E}}^\e \left| G^{\theta_\eta}_{\xi,n_\eta}(s, {u}^{\e,\d}(s), {v}^{\e,\d}(s)) - \bar{G}^{\theta_\eta}_{\xi,n_\eta} (s, {u}^{\e,\d}(s))\right| ds=:\sum_{k=1}^3 J_{\e, \eta, k}.} 
\end{array}
\end{equation}
Since both $G^{\theta_\eta}_{\xi,n_\eta}$ and $\bar{G}^{\theta_\eta}_{\xi,n_\eta}$ are Lipschitz continuous, by Lemma \ref{l4.2}, we have
\begin{equation}
\label{sy85}
\begin{array}{l}
\ds{J_{\e, \eta, 1}+J_{\e, \eta, 2}}\\[10pt]
\ds{\leq \left([G^{\theta_\eta}_{\xi,n_\eta}]_{\text{\tiny{Lip}}}+[\bar{G}^{\theta_\eta}_{\xi,n_\eta}]_{\text{\tiny{Lip}}}\right)\sum_{k=0}^{[T/\d]+1} \int_{k\d}^{(k+1)\d} {\mathbb{E}}^\e\left(\Vert {u}^{\e}(s)-{u}^{\e, \delta}(s)\Vert+\Vert {v}^{\e}(s)-{v}^{\e,\d}(s)\Vert\right)\,ds}\\[14pt]
\ds{ \leq c_{u,\eta}(u_0,v_0) \sum_{k=0}^{[T/\d]+1} \int_{k\d}^{(k+1)\d} \left( \rho_0(s,k\d)^{1/2} + \left(\,\frac 1{\e} \exp\left(\frac{c\,\d}{\e}\right)\int_{k\d}^s \rho_0(\zeta,k\d) d\zeta\right)^{1/2}\right) ds.}
\end{array}
\end{equation}
Moreover, we have
\begin{equation}
\begin{array}{l}
\ds{J_{\e, \eta, 3}= \sum_{k=0}^{[T/\d]+1}\int_{k\d}^{(k+1)\d} {\mathbb{E}}^\e \left| G^{\theta_\eta}_{\xi,n_\eta}(s, {u}^{\e}(k\d), {v}^{\e,\d}(s)) - \bar{G}^{\theta_\eta}_{\xi,n_\eta} (s, {u}^{\e}(k\d)) \right|ds} \\[18pt]
\ds{\leq\sum_{k=0}^{[T/\d]+1}\int_{k\d}^{(k+1)\d} {\mathbb{E}}^\e \left| G^{\theta_\eta}_{\xi,n_\eta}(s, {u}^{\e}(k\d), {v}^{\e,\d}(s)) - G^{\theta_\eta}_{\xi,n_\eta}(k\d, {u}^{\e}(k\d), {v}^{\e,\d}(s))\right| ds} \\[18pt]
\ds{+  \sum_{k=0}^{[T/\d]+1}\int_{k\d}^{(k+1)\d} {\mathbb{E}}^\e  \left|G^{\theta_\eta}_{\xi,n_\eta}(k\d , {u}^{\e}(k\d), {v}^{\e,\d}(s)) -\bar{G}^{\theta_\eta}_{\xi,n_\eta} (k\d, {u}^{\e}(k\d))\right| ds }\\[18pt]
\ds{+  \sum_{k=0}^{[T/\d]+1}\int_{k\d}^{(k+1)\d} {\mathbb{E}}^\e  \left|\bar{G}^{\theta_\eta}_{u,n_\eta} (k\d, {u}^{\e}(k\d))-\bar{G}^{\theta_\eta}_{u,n_\eta} (s, {u}^{\e}(k\d)) \right| ds =:\sum_{k=1}^3 K_{\e, \eta, k}.} 
\end{array}
\end{equation}
Thanks again to the Lipschitz continuity of $\bar{G}^{\theta_\eta}_{\xi,n_\eta}$ and $G^{\theta_\eta}_{\xi,n_\eta}$, there exists some constant $c_{\eta}>0$ such that
\begin{equation}
\label{sy86}
\begin{aligned}
K_{\e, \eta, 1}+K_{\e, \eta,3} \leq c_{u, \eta}  \sum_{k=0}^{[T/\d]+1} \int_{k\d}^{(k+1)\d}(s-k\d)\, ds \leq c_{\eta}\,T \d.
\end{aligned}
\end{equation}
Finally, by proceeding as in  \cite[Lemma 2.3]{SC1} and \cite[Theorem 3.5]{SCMF1}, as a consequence of \eqref{sy30} we have
\begin{equation}
\label{sy87}
\begin{aligned}
 K_{\e, \eta, 2}
= \sum_{k=0}^{[T/\d]+1}\int_{k\d}^{(k+1)\d} {\mathbb{E}}^\e  \left|G^{\theta_\eta}_{\xi,n_\eta}(k\d , {u}^{\e}(k\d), {v}^{\e,\d}(s)) -\bar{G}^{\theta_\eta}_{\xi,n_\eta} (k\d, {u}^{\e}(k\d))\right| ds
\leq  c_{\eta}\,\sqrt{\frac{\e}{\d}}.
\end{aligned}
\end{equation}

Therefore, in view of \eqref{eq4.3}, from \eqref{sy81},  \eqref{sy85}, \eqref{sy86} and \eqref{sy87}, we obtain
\begin{equation}
\begin{array}{l}
\ds{ {\mathbb{E}}^\e\sup_{t \in [0,T]}\left|\int_0^{T} \langle  F^1(s, u^\e(s), v^\e(s)) , \xi(s) \rangle ds- \int_0^{t} \langle  \bar{F}^1(s, {u}^{\e}(s)) , \xi(s) \rangle ds \right|}\\[14pt]
\ds{\leq  c_{\eta}(x_0) \sum_{k=0}^{[T/\d]+1} \int_{k\d}^{(k+1)\d} \left( \rho_0(t,k\d)^{1/2} + \left(\,\frac 1{\e} \exp\left(\frac{c\,\d}{\e}\right)\int_{k\d}^t \rho_0(t,k\d) ds\right)^{1/2}\right) dt }\\[14pt]
\ds{+I_{\e, 3}(\theta_\eta, n_\eta)+c_{ \eta}\,T \d+c_{\eta}\,\sqrt{\frac{\e}{\d}}+\eta.}
\end{array} 
\end{equation}
Under the assumption  that $u_0 \in D((-A_1)^{\beta})$, for some $\beta>0$, according to  \eqref{l1.1eq3}, there exists some $\lambda>0$ such that
\[
\rho_u(t,s) \leq c(x_0)(t-s) ^{\lambda},
\]
so that
\[
\begin{array}{l}
\ds{\sum_{k=0}^{[T/\d]+1} \int_{k\d}^{(k+1)\d} \left( \rho_0(t,k\d)^{1/2} + \left(\frac 1{\e} \exp\left(\frac {c\, \d}{\e}\right)\int_{k\d}^s \rho_u(\zeta,k\d) d\zeta\right)^{1/2}\right) ds}\\[18pt]
\ds{\leq c\sum_{k=0}^{[T/\d]+1} \int_{k\d}^{(k+1)\d} \left( \d^{\lambda/2} + \left(\frac 1{\e} \exp\left(\frac {c\, \d}{\e}\right)\d^{1+\lambda}\right)^{1/2}\right) ds }\\[18pt]
\ds{\leq c\, t \left( \d^{\lambda/2} + \frac 1{\e^{1/2}} \exp\left(\frac {c\, \d}{2\,\e}\right)\d^{(1+\lambda)/2}\right).}
\end{array}
\]
Therefore, we have
\begin{equation}
\label{sy92}
\begin{array}{l}
\ds{ {\mathbb{E}}^\e\sup_{t \in [0,T]}\left|\int_0^{t} \langle  F^1(s, u^\e(s), v^\e(s)) , \xi(s) \rangle ds- \int_0^{t} \langle  \bar{F}^1(s, {u}^{\e}(s)) , \xi(s) \rangle ds \right|}\\[14pt]
\ds{\leq c\, T \left( \d^{\lambda/2} + \frac 1{\e^{1/2}} \exp\left(\frac {c\, \d}{2\,\e}\right)\d^{(1+\lambda)/2}\right)+I_{\e, 3}(\theta_\eta, n_\eta)+c_{ \eta}\,T \d+c_{ \eta}\,\sqrt{\frac{\e}{\d}}+\eta.}
\end{array}\end{equation}

Now if we take
\[
\d_{\e}:= \frac{2}c\,\e |\ln \e|^{\lambda/2},
\] 
we have 
\begin{equation}
\label{sy90}	
\lim_{\e\to 0} \d_\e=0,\ \ \ \ \ 
\lim_{\e\to 0}\frac{\e}{\d_\e} = \frac c2\,\lim_{\e \to 0}|\ln \e|^{-\lambda/2}=0.
\end{equation}
Moreover, we have
\[
 \frac 1{\e^{1/2}} \exp\left(\frac {c\, \d_\e}{2\,\e}\right)\d_\e^{(1+\lambda)/2}= \e^{\eta}  |\ln \e|^{-\eta(1+ \eta)/2} \left(\frac{1}{\e}\right)^{\eta/2}=\e^{\eta/2}  |\ln \e|^{-\eta(1+ \eta)/2},
\]
so that
\begin{equation}
\label{sy91}	
\lim_{\e\to 0} \frac 1{\e^{1/2}} \exp\left(\frac {c\, \d_\e}{2\,\e}\right)\d_\e^{(1+\lambda)/2}=0.
\end{equation}
Hence, due to  \eqref{sy90} and \eqref{sy91}, from \eqref{sy92} we get
\[
\limsup_{\e\to 0} {\mathbb{E}}^\e\sup_{t \in [0,T]}\left|\int_0^{t} \langle  F^1(s, u^\e(s), v^\e(s)) , \xi(s) \rangle ds- \int_0^{t} \langle  \bar{F}^1(s, {u}^{\e}(s)) , \xi(s) \rangle ds \right|\leq \eta,
\]
and the arbitrariness of $\eta>0$ allows to obtain \eqref{eq aux2.1}. 
\end{proof}

\subsection{Conclusion}
Let $\{\e_m\}_{m \in\,\mathbb{N}}$ a sequence in $(0,1)$ converging to $0$ such that
\begin{equation}
	\label{fine1}
	\lim_{m\to\infty} \mathcal{L}(u^{\e_m})=\mu_{\text{\tiny{slow}}},\ \ \ \ \ \text{weakly in}\ C([0,T];H).
\end{equation}
In particular  the sequence 
$\left\{ (u^{\e_m},W^{\e_m})\right\}_{m \in\,\mathbb{N}}\subset C([0,T];H)\times C([0,T];\mathcal{D}^\prime)$ is tight. Therefore there exists a subsequence of $\{\e_m\}_{m \in\,\mathbb{N}}$ that we will still denote $\{\e_m\}_{m \in\,\mathbb{N}}$ and $\nu \in\,\mathcal{P}(C([0,T];H)\times C([0,T];\mathcal{D}^\prime))$ such that 
\[\lim_{m\to\infty} \mathcal{L}(u^{\e_m},W^{\e_m})=\nu,\ \ \ \ \ \text{weakly in}\ C([0,T];H)\times  C([0,T];\mathcal{D}^\prime).\]

Now, as a consequence of the Skorokhod theorem, we can fix a probability space $(\bar{\Omega}, \bar{\mathcal{F}}, \bar{\mathbb{P}})$, a sequence $\{(u^m, W^m)\}_{m \in\,\mathbb{N}}\subset C([0,T];H)\times C([0,T];\mathcal{D}^\prime)$ and  $(\bar{u}, \bar{W})\in  C([0,T];H)\times C([0,T];\mathcal{D}^\prime)$ such that
\begin{equation}
	\label{fine3}
	\mathcal{L}(u^m, W^m)=\mathcal{L}(u^{\e_m}, W^{\e_m}),\ \ \ \ \ \mathcal{L}(\bar{u}, \bar{W})=\nu,
\end{equation}
and
\begin{equation}
	\label{fine4}
	\lim_{m\to\infty} (u^m, W^m)=(\bar{u}, \bar{W}),\ \ \ \ \ \bar{\mathbb{P}}-\text{a.s.}\ \ \text{in}\ C([0,T];H)\times C([0,T];\mathcal{D}^\prime).
\end{equation}
Notice that, due to \eqref{fine1}, we have $\Pi_1\nu=\mu_{\text{\tiny{slow}}}$. Moreover, since all $W^\e$ are space-time white noises in $H$, we have that $\Pi_2 \nu$ is a space-time white noise. All this implies that 
$\mathcal{L}(\bar{u})=\mu_{\text{\tiny{slow}}}$ and $\bar{W}$ is a space-time white noise.

Next, for every $\e \in\,(0,1)$ and $\xi \in\,C^1([0,T];D(A_1^\star))$, we define 
\[R^\e_{\xi}(t):=\langle u^\e(t),\xi(t)\rangle -
\langle u_0,\xi(0)\rangle-\int_0^t \langle \bar{F}_1(s,u^\e(s)),\xi(s)\rangle\,ds-\int_0^t\langle Q_1 dW^\e(s),\xi(s)\rangle.\]
According to Lemma \ref{l aux2}, we have that 
\begin{equation}
\label{fine2}
\lim_{\e \to 0}\,\mathbb{E}^\e\sup_{t \in\,[0,T]} \vert R^\e_\xi(t)\vert=0.	
\end{equation}
Therefore, if we define
\begin{equation}
\label{fine5}	
R^m_{\xi}(t):=\langle u^m(t),\xi(t)\rangle -
\langle u_0,\xi(0)\rangle-\int_0^t \langle \bar{F}_1(s,u^m(s)),\xi(s)\rangle\,ds-\int_0^t\langle Q_1 dW^m(s),\xi(s)\rangle,
\end{equation}
due to \eqref{fine3} we have
\[\lim_{m\to\infty}\,\bar{\mathbb{E}}\sup_{t \in\,[0,T]} \vert R^m_\xi(t)\vert=0.	\]
This, together with \eqref{fine4}, implies that if we take the limit as $m\to \infty$ in both sides of \eqref{fine5}, 
and if we show that
\begin{equation}
\label{fine6}
\lim_{m\to\infty}\bar{\mathbb{E}}\sup_{t \in\,[0,T]}\left\vert \int_0^t \langle \bar{F}_1(s,u^m(s)),\xi(s)\rangle\,ds-\int_0^t \langle \bar{F}_1(s,\bar{u}(s)),\xi(s)\rangle\,ds\right\vert =0,	
\end{equation}
 we obtain that
$\bar{u}$ satisfies the equation
\[\langle \bar{u}(t),\xi(t)\rangle =
\langle u_0,\xi(0)\rangle+\int_0^t \langle \bar{F}_1(s,\bar{u}(s)),\xi(s)\rangle\,ds-\int_0^t\langle Q_1 d\bar{W}(s),\xi(s)\rangle.\]

By using the notations introduced in Subsection \ref{ssec5.2}, we can rewrite \eqref{fine6} as
\begin{equation}\label{fine7}
\lim_{m \to \infty} \bar{\mathbb{E}} \sup_{t \in [0,T]} \left| \int_0^t \left[\bar{G}_\xi(s, {u}^{m}(s))- \bar{G}_\xi(s, \bar{u}(s))\right] ds\right|=0.
\end{equation}
We have
\[\begin{array}{l}
\ds{\bar{\mathbb{E}} \sup_{t \in [0,T]} \left| \int_0^t \left[\bar{G}_\xi(s, {u}^{m}(s))- \bar{G}_\xi(s, \bar{u}(s))\right] ds\right|\leq \int_0^T\bar{\mathbb{E}}\left\vert \bar{G}_\xi(s, {u}^{m}(s))-\bar{G}^\theta_\xi(s, {u}^{m}(s))\right\vert\,ds}	\\[18pt]
\ds{+\int_0^T\bar{\mathbb{E}}\left\vert \bar{G}^\theta_\xi(s, {u}^{m}(s))-\bar{G}^\theta_\xi(s, \bar{u}(s))\right\vert\,ds+\int_0^T\bar{\mathbb{E}}\left\vert \bar{G}^\theta_\xi(s, \bar{u}(s))-\bar{G}_\xi(s, \bar{u}(s))\right\vert\,ds=:\sum_{k=1}^3 I^\theta_{m,k}.}
\end{array}\]
According to \eqref{eqG1}, we have
\[I_{m,1}^\theta+I_{m,3}^\theta\leq c_T\,\theta \int_0^T\left [\bar{\mathbb{E}}\bar{V}(u^m(t))+\bar{\mathbb{E}}\bar{V}(\bar{u}(t))\right]\,dt.\]
As a consequence of \eqref{fine3} and \eqref{sy20}, we have
\[\sup_{m \in\,\mathbb{N}}\int_0^T \bar{\mathbb{E}}\,\bar{V}(u^m(t))\,dt=\sup_{m \in\,\mathbb{N}}\int_0^T \mathbb{E}^{\e_m}\bar{V}(u^{\e_m}(t))\,dt<\infty.\]
Moreover, Fatou's lemma implies
\[\begin{array}{ll}
\ds{\int_0^T \bar{\mathbb{E}}\ \bar{V}(\bar{u}(t))\,  dt}  &  \ds{\leq \liminf_{R\to\infty}\int_0^T \bar{\mathbb{E}} \left(\bar{V}(\bar{u}(t))\wedge R\right)\, dt}\\[14pt]
&\ds{=\liminf_{R\to\infty}\lim_{m\to\infty} \int_0^T \bar{\mathbb{E}} \left(\bar{V}(\bar{u}(t))\wedge R\right)\, dt \leq \sup_{m  \in\,\mathbb{N}}\int_0^T \bar{\mathbb{E}}\, \bar{V}(u^{m}(t))\, dt<\infty.}
	\end{array}\]
Therefore, for every $\eta>0$, we can find $\theta_\eta>0$ such that
\[\sup_{m \in\,\mathbb{N}}\left(I_{m,1}^{\theta_\eta}+I_{m,3}^{\theta_\eta}\right)<\eta.\]
In particular, due to the arbitrariness of $\eta>0$, \eqref{fine7} follows once we show 
that
\[\lim_{m\to\infty}\int_0^T\bar{\mathbb{E}}\left\vert \bar{G}^{\theta_\eta}_\xi(s, {u}^{m}(s))-\bar{G}^{\theta_\eta}_\xi(s, \bar{u}(s))\right\vert\,ds=0.\]
But this is clearly a consequence of \eqref{fine4}, as $\bar{G}^{\theta_\eta}_\xi$ is bounded and continuous.

\section{An example}
\label{sec7}
In this section, we are going to discuss some examples of slow-fast systems of SPDEs which satisfy the conditions we have assumed in the previous sections and hence can be treated with our theory.

Let $\mathcal{O}\subset \mathbb{R}^d$ be a bounded open set, with a smooth boundary and let $H$ be the Hilbert space $L^2(\mathcal{O})$, endowed with the norm $\Vert\cdot\Vert$ and the scalar product $\langle\cdot,\cdot\rangle$. As in Section \ref{sec2}, we denote by $\mathcal{H}_2$ the product sopace $H\times H$.

\subsection{About Hypothesis \ref{h1}} 
We assume that for every $p \in\,(1,\infty)$ the operators $(A_{1, p},D(A_{1, p}))$ and $(A_{2, p},D(A_{2, p}))$ are the realizations in $L^p(\mathcal{O})$ of second order uniformly elliptic operators $\mathcal{A}_i$, $i=1, 2$, satisfying Hypothesis \ref{h1}. All these operators are consistent, in the sense that for every $p, r \in\,(0,\infty)$
\[A_{i, p} x=A_{i, r} x,\ \ \ \ x \in\,D(A_{i, p})\cap D(A_{i, r}),\ \ \ i=1, 2,\]
and for this reason,   whenever there is no room for confusion, for the sake of simplicity we shall denote all of them as $(A_{1},D(A_{1}))$ and $(A_{2},D(A_{2}))$.
In what follows, we will use the fact that for every $p \in\,(1,\infty)$ there exists $\la_p>0$ such that
\[\langle A_i x,x |x|^{p-2}\rangle \leq -\la_p\,\Vert(-A_i)^{1/2}(x|x|^{\frac{p-2}2})\Vert^2,\ \ \ \ x \in\,D(A_{i,p}),\ \ \ i=1,\,2.
\]
In particular, this implies that 
\begin{equation}
\label{sy600}
\langle A_i x,x |x|^{p-2}\rangle \leq -\a_{1, i}\,\la_p\,\Vert x\Vert_{L^p}^p,\ \ \ \ x \in\,D(A_{i,p}),\ \ i=1, 2,
\end{equation}
(for all details, we refer e.g. to \cite{lunardi}).

As far as the Gaussian perturbations in the slow and in the fast equations are concerned, we assume that the operators $Q_{1}$ and $Q_2$ satisfy Hypothesis \ref{h1}. Moreover, we assume that one of the following two conditions are satisfied:
\begin{enumerate}
\item[C1.] both $Q_1$ and $Q_2$ are invertible, with $Q_1^{-1}, Q_2^{-1} \in\,\mathcal{L}(H)$,
\item[C2.] both $Q_1$ and $Q_2$ are Hilbert-Schmidt operators.	
\end{enumerate}

Under reasonable conditions of the coefficients of the differential operators $\mathcal{A}_i$ and the domain $\mathcal{O}$, we have that 
\[\alpha_{i, k}\sim k^{\frac d2},\ \ \ \ k \in\,\mathbb{N}.\]
Hence, Condition C1. is satisfied only if $d=1$. In particular, if we want to treat the case of SPDEs of reaction-diffusion type in space dimension $d>1$, we can only consider Gaussian noise with trace-class covariance.

As we did in Section \ref{sec2}, for every $\e \in\,(0,1)$ we define
\begin{equation}
	\label{sy621}
	A^{\e}(x_1,x_2) = (A_1x_1, \e^{-1} A_2x_2), \ \ \ \ (x_1, x_2) \in\,D(A^\e):=D(A_1)\times D(A_2)\subset \mathcal{H}_2,
\end{equation}
and 
\begin{equation}
\label{sy622}	
Q^{\e}(x_1,x_2) = (Q_1x_1, \e^{-1/2} Q_2x_2),\ \ \ \ (x_1,x_2) \in \mathcal{H}_2.
\end{equation}
Notice that if Condition C1. holds, then 
$Q^\e $ is invertible with $(Q^\e)^{-1} \in\,\mathcal{L}(\mathcal{H}_2)$, for each $\e \in\,(0,1)$. In the same way, if Condition C2. holds, then $\text{Tr}\,[(Q^\e)^2]<\infty$, for each $\e \in\,(0,1)$.

\subsection{About Hypothesis \ref{h2}} We assume that $b:[0,T]\times\mathcal{O}\times \mathbb{R}^2\to\mathbb{R}$ is a continuous function satisfying the following conditions.
\begin{enumerate}
\item There exist  $c_1>0$ and $m_1, m_2 \geq 1$ and  a function $a_1 \in\,L^4(\mathcal{O})$ such that 
\begin{equation}
\label{sy101}
\sup_{t \in\,[0,T]}\,|b(t,\xi,	\si,\la)|\leq c_1\,\left(a_1(\xi)+|\si|^{m_1}+|\la|^{m_2}\right),
\end{equation}
for all $\xi \in\,\mathcal{O}$ and $\si, \la \in\,\mathbb{R}$.
\item There exist $c_2>0$, $\kappa_1\leq 2 m_2$ and $ \kappa_2 \geq 0$ and a function $a_2 \in\,L^1(\mathcal{O})$ such that
\begin{equation}	
\label{sy100}
\sup_{t \in\,[0,T]}b(t,\xi,\si+\rho,\la)\si\leq c_2\left(a_2(\xi)+|\si|^2+|\la|^{\kappa_1}+|\rho|^{\kappa_2}	\right),
\end{equation}
for all $\xi \in\,\mathcal{O}$ and $\si, \rho, \la \in\,\mathbb{R}$.

\end{enumerate}

Now, given any two functions $x, y:\mathcal{O}\to \mathbb{R}$, we define
\begin{equation}
\label{sy102}
F_1(t,x,y)(\xi)=b(t,\xi,x(\xi),y(\xi)),\ \ \ \ t \in\,[0,T],\ \ \xi \in\,\mathcal{O}.	
\end{equation}
According to \eqref{sy101}, 
if we define 
\[D(F_1):=L^{2m_1}(\mathcal{O})\times L^{2m_2}(\mathcal{O}),\]
 we have that $F_1:[0,T]\times D(F_1)\to H$ and
\[\Vert F_1(t,x,y)\Vert \leq c\left (1+\Vert x\Vert _{L^{2m_1}}^{m_1}+\Vert y\Vert_{L^{2m_2}}^{m_2}\right),\ \ \ \ (x,y) \in\, D(F_1).\]
Moreover, since we are assuming $\kappa_1\leq 2 m_2$, according to \eqref{sy100}, we have that
\begin{equation}
\label{sy103}
\ds{\langle F_1(t,x+z,y),x\rangle\leq c\left(1+\Vert x\Vert^2+\Vert y\Vert^{\kappa_1}_{L^{\kappa_1}}+\Vert z\Vert^{\kappa_2}_{L^{\kappa_2}}\right),	}
\end{equation}
for every $(x,y) \in\, D(F_1)$ and $z \in\,L^{\kappa_2}(\mathcal{O})$.

Next, for every $\theta \in\,(0,1)$, we define 
\[b_\theta(t,\xi,\si,\la):=\frac{b(t,\xi,	\si,\la)}{1+\theta\,|b(t,\xi,	\si,\la)|},\ \ \ \ \ t \in\,[0,T],\ \ \ \xi \in\,\mathcal{O},\ \ \ (\si,\rho) \in\,\mathbb{R}^2.\]
 The function $b_\theta:[0,T]\times \mathcal{O}\times \mathbb{R}^2\to \mathbb{R}$ is continuous and bounded. Therefore, if for every $x, y \in\,H$ we define
 \[F_1^\theta(t,x,y)(\xi):= b_\theta(t,\xi,x(\xi),y(\xi)),\ \ \ \ t \in\,[0,T],\ \ \xi \in\,\mathcal{O},\]
 we have that $F_1^\theta:[0,T]\times \mathcal{H}_2\to H$ is continuous and bounded.
 Since 
\[|b_\theta(t,\xi,\si,\la)|\leq |b(t,\xi,\si,\la)|,\] 
according to \eqref{sy101} we have
\begin{equation}
\label{sy105}
\begin{array}{ll}
\ds{\Vert F_1^\theta(t,x,y)\Vert^2 }  &  \ds{\leq \Vert F_1(t,x,y)\Vert^2\leq c\int_\mathcal{O}\le(|a_2(\xi)|^2+|x(\xi)|^{2m_1}+|y(\xi)|^{2m_2}\right)\,d\xi}\\[14pt]
& \ds{\leq c\,\left(1+\Vert x\Vert_{L^{2m_1}}^{2m_1}+\Vert y\Vert_{L^{2m_2}}^{2m_2}\right).}
\end{array}	
\end{equation}
 Moreover, since 
 \[|b_\theta(t,\xi,\si,\la)-b(t,\xi,\si,\la)|=\theta\,|b(t,\xi,\si,\la)|^2,\] 
thanks again to \eqref{sy101} we have
 \begin{equation}  \label{sy106}
 \begin{array}{l}
\ds{\Vert F^\theta_1(t,x,y)-F_1(t,x,y)\Vert^2 \leq \theta^2 \int_{\mathcal{O}}|b(t,\xi,x(\xi),y(\xi))|^4\,d\xi}\\[10pt]
\ds{\leq c\,\theta^2\int_\mathcal{O}\left(|a_2(\xi)|^4+|x(\xi)|^{4m_1}+|y(\xi)|^{4m_2}\right)\,d\xi\leq c\,\theta^2\left(1+\Vert x\Vert_{L^{4m_1}}^{2m_1}+\Vert y\Vert_{L^{4m_2}}^{2m_2}\right)^2.	}
\end{array}
\end{equation}
This implies that if we define
\begin{equation}
\label{sy120}	
V(x,y):=c_V\left(1+\Vert x\Vert_{L^{4m_1}}^{2m_1}+\Vert y\Vert_{L^{4m_2}}^{2m_2}+\Vert y\Vert_{L^{2\kappa_1 m_1}}^{\kappa_1 m_1}\right),
\end{equation}
for some $c_V>0$ sufficiently large, as a consequence of \eqref{sy105} and \eqref{sy106} we obtain
\begin{equation} \label{1.4-tris}
\Vert F_1^\theta(t,x,y)\Vert^2\leq \Vert F_1(t,x,y)\Vert^2\leq V(x,y),	
\end{equation}
and
\begin{equation} \label{1.4-bis}
\Vert F^\theta_1(t,x,y)-F_1(t,x,y)\Vert\leq \theta\,V(x,y).	
\end{equation}
The function $V:\mathcal{H}_2\to [0,+\infty]$ is convex and lower semi-continuous and 
\[D_V:=\{V<\infty\}\subset D\times H.\]
\begin{Remark}
{\em In \eqref{1.4-bis} and \eqref{1.4-tris} we only need
\[V(x,y):=c_V\left(1+\Vert x\Vert_{L^{4m_1}}^{2m_1}+\Vert y\Vert_{L^{4m_2}}^{2m_2}\right).\]
The reason why we define $V$ as in \eqref{sy120} is because we will need it when we will prove that condition \eqref{eq h4} holds.
}
\end{Remark}

\subsection{About Hypothesis \ref{h3}} We introduce a function $g:[0,T]\times \mathcal{O}\times \mathbb{R}^2\to\mathbb{R}$ continuous, such that $g(t,\xi,\cdot):\mathbb{R}^2\to\mathbb{R}$ is Lipschitz continuous, uniformly with respect to $(t,\xi) \in\,[0,T]\times \mathcal{O}$. We assume that there exists $L_2<\alpha_{2, 1}$ that
\begin{equation}
\label{sy620}
\sup_{{\substack{(t,\xi) \in\,[0,T]\times \mathcal{O}\\\ \rho \in\,\mathbb{R}}}}	|g(t,\xi,\rho,\si_1)-g(t,\xi,\rho,\si_2)|\leq L_2\,|\si_1-\si_2|.
\end{equation}
Thus, if we define
\[F_2(t,x)(\xi):=g(t,\xi,x(\xi)),\ \ \ \ \xi \in\,\mathcal{O},\]
for every $x \in\,\mathcal{H}_2$, we have that $F_2$ satisfies Hypothesis \ref{h3}.

\subsection{About Hypothesis \ref{h4}}

As we did in Section \ref{sec2}, for every $\e \in\,(0,1)$ we define $D(F^\e):=D(F_1)\subseteq \mathcal{H}_2$ and
\[
F^{\e}(t,x):=\left(F_1(t,x) ,\e^{-1} F_2(x)\right),\ \ \ \ t \in\,[0,T],\ \ x \in\,D(F^\e).
\]
Moreover, for every
 $\theta \in\,(0,1)$ we define
\[F^{\e,\theta}(t,x):=\left(F^\theta_1(t,x) ,\e^{-1} F_2(x)\right),\ \ \ \ t \in\,[0,T],\ \ x \in\,\mathcal{H}_2.\]
Since $F^\theta_1$ and $F_2$ are both continuous, the mapping $F^{\e, \theta}(t,\cdot):\mathcal{H}_2\to\mathcal{H}_2$ is continuous.
Moreover, due to the linear growth of  $F_2(t,\cdot)$, which is uniform with respect to $t \in\,[0,T]$, according to \eqref{1.4-bis}
we have
\[\begin{array}{ll}
\ds{\Vert F^{\e, \theta}(t,x)\Vert_{\mathcal{H}_2}^2}  &  \ds{=\Vert F^\theta_1(t,x)\Vert^2+\e^{-2}\,\Vert F_2(t,x)\Vert^2\leq \Vert F_1(t,x)\Vert^2+\e^{-2}\,\Vert F_2(t,x)\Vert^2}\\[14pt]
&  \ds{\leq V(x)+\frac c{\e^2}\left(1+\Vert x\Vert_{\mathcal{H}_2}^2\right)=:V_\e(x),}
	\end{array}
\]
and the function $V_\e:\mathcal{H}_2\to[0,+\infty]$ is convex and lower semi-continuous.
Finally,
\[\begin{array}{ll}
\ds{\Vert F^{\e, \theta}(t,x)-F^\e(t,x)\Vert_{\mathcal{H}_2}}  &  \ds{=\Vert F^\theta_1(t,x)-F_1(t,x)\Vert\leq \theta\,V(x)\leq \theta\,V_\e(x).}
	\end{array}
\]

The operators $(A^\e, D(A^\e))$ and $Q^\e$ and the functions $F^{\e}$, $F^{\e, \theta}$ and $V_\e$ that we have just constructed satisfy the all the conditions assumed Appendix A, for each $\e \in\,(0,1)$. Therefore,  once we have fixed $(u_0,v_0) \in\,D_V$ for every $\e \in\,(0,1)$ there exists a martingale solution 
\[({\Omega}^{\e}, {\mathcal{F}}^{\e}, \{{\mathcal{F}}_t^{\e}\}_{t \in\,[0,T]}, {\mathbb{P}}^{\e},\{{W}^{\e}_t\}_{t \in\,[0,T]}, \{X^\e(t)\}_{t \in\,[0,T]}),\]
for system \eqref{eq spde}. Now, we check that condition \eqref{eq h4} holds.

\begin{Lemma}
\label{lemma7.2}
Assume that the function $a_2$ introduced in \eqref{sy100} belongs to $L^{2m_1}(\mathcal{O})$ and there exists $c_T>0$ such that
\begin{equation}
\label{sy624}
\sup_{t  \in\,[0,T]}\mathbb{E}^\epsilon\,\Vert W_{A_1}(t)\Vert_{L^{\bar{p}}}^{\bar{p}}+\sup_{t \in\,[0,T]}\mathbb{E}^\epsilon\,\Vert W_{A_2}(t)\Vert_{L^{\bar{q}}}^{\bar{q}}\leq c_T,\ \ \ \ \ \e \in\,(0,1),
\end{equation}
where $\bar{p}:=2 \kappa_2 m_1$ and $\bar{q}:= 2 \kappa_1 m_1\vee 4 m_2$ and where $\kappa_2$, $\kappa_1$, $m_1$ and $m_2$ are  the constants introduced in \eqref{sy101} and \eqref{sy100}. Moreover, assume that
\begin{equation}
\label{sy638}
\sup_{(t,\xi) \in\,[0,T]\times \mathcal{O}}|g(t,\xi,\si,\rho)|\leq c\left(1+|\si|^{\frac{2}{\kappa_1}\wedge \frac {m_1}{m_2}\wedge 1}+|\rho|\right),\ \ \ \ \ (\si,\rho) \in\,\mathbb{R}^2,	
\end{equation}
and
\begin{equation}
\label{sy630}
L_2<\a_{2, 1}\lambda_{\bar{q}},
\end{equation}
where $L_2$ is the constant introduced in \eqref{sy620}. Then
condition \eqref{eq h4} is satisfied.
\end{Lemma}
\begin{proof}
If we define $\Gamma^\e_1(t):= {u}^{\e}(t)-{W}^{\e}_{A_1}(t)$, due to \eqref{sy600}, we have
\[\begin{array}{l}
\ds{\frac{1}{4m_1}\frac{d}{dt} \|\Gamma^\e_{1}(t)\|^{4m_1}_{L^{4m_1}} =\langle A_1 \Gamma^\e_1(t)+ F_1(t,\Gamma^\e_1(t) + {W}^{\e}_{A_1}(t), {v}^{\e}(t)) , \Gamma^\e_1(t)\,|\Gamma^\e_1(t)|^{4m_1-2}\rangle}\\[14pt]
\ds{\quad \quad \leq -\a_{1,1}\lambda_{4m_1}\,\Vert \Gamma^\e_1(t)\Vert^{4m_1}_{L^{4m_1}}+\langle F_1(t,\Gamma^\e_1(t) + {W}^{\e}_{A_1}(t), {v}^{\e}(t)) , \Gamma^\e_1(t)\,|\Gamma^\e_1(t)|^{4m_1-2}\rangle.}
\end{array}
\]
Thanks to \eqref{sy100}, we  have 
\[\begin{array}{l}
\ds{ \langle F_1(t,\Gamma^\e_1(t) + {W}^{\e}_{A_1}(t), {v}^{\e}(t)),  \Gamma_1^{\e}(t) \Gamma^\e_1(t)\,|\Gamma^\e_1(t)|^{4m_1-2}\rangle}\\[14pt]
\ds{\leq \int_{\mathcal{O}}\left(a_2(\xi) + |\Gamma^\e_1(t,\xi)|^2+|{v}^{\e}(t,\xi)|^{\kappa_1}+|{W}^{\e}_{A_1}(t,\xi)|^{\kappa_2}\right)\,|\Gamma^\e_1(t,\xi)|^{4m_1-2}\,d\xi}\\[14pt]
\ds{\leq \int_{\mathcal{O}}\left(|a_2(\xi)|^{2m_1} + |\Gamma^\e_1(t,\xi)|^{4m_1}+|{v}^{\e}(t,\xi)|^{\bar{q}}+|{W}^{\e}_{A_1}(t,\xi)|^{\bar{p}}\right)\,d\xi,}
\end{array}
\]
so that
\[\begin{array}{ll}
\ds{\frac{d}{dt} \|\Gamma^\e_{1}(t)\|^{4m_1}_{L^{4m_1}} } & \ds{\leq c\,\Vert \Gamma^\e_1(t)\Vert^{4m_1}_{L^{4m_1}}+c\left(1+\Vert {W}^{\e}_{A_1}(t)\Vert_{L^{\bar{p}}}^{\bar{p}}+\Vert{v}^{\e}(t)\Vert^{\bar{q}}_{L^{\bar{q}}}\right).}
\end{array}
\]
By comparison, we obtain
\[\mathbb{E}^\epsilon\,\|\Gamma^\e_{1}(t)\|^{4m_1}_{L^{4m_1}} \leq c_T\left(1+\Vert u_0\Vert^{4m_1}_{L^{4m_1}}+\sup_{t \in\,[0,T]}\mathbb{E}^\epsilon\,\Vert {W}^{\e}_{A_1}(t)\Vert_{L^{\bar{p}}}^{\bar{p}}+\int_0^T\mathbb{E}^\epsilon\Vert{v}^{\e}(t)\Vert^{\bar{q}}_{L^{\bar{q}}}\,dt\right),\]
and, in view of \eqref{sy624}, this yields
\begin{equation} \label{sy633}
\mathbb{E}^\epsilon\,\|u^\e(t)\|^{4m_1}_{L^{4m_1}} \leq c_T\left(1+\Vert u_0\Vert^{4m_1}_{L^{4m_1}}\right)+c_T\int_0^t\mathbb{E}^\epsilon\,\Vert{v}^{\e}(s)\Vert^{\bar{q}}_{L^{\bar{q}}}\,ds,\ \ \ \ t \in\,[0,T].	
\end{equation}

As we have done above for $u^\e$, we denote $\Gamma^\e_2(t):= {v}^{\e}(t)-{W}^{\e}_{A_2}(t)$. Due to \eqref{sy600},  we have
\begin{equation}  \label{sy626}
\begin{array}{l}
\ds{\frac{1}{\bar{q}}\frac{d}{dt} \|\Gamma^\e_{2}(t)\|^{\bar{q}}_{L^{\bar{q}}} =\frac 1\e \langle A_2 \Gamma^\e_2(t)+ F_2(t,u^\e(t),\Gamma^\e_2(t) + {W}^{\e}_{A_2}(t)) , \Gamma^\e_2(t)\,|\Gamma^\e_2(t)|^{\bar{q}-2}\rangle}\\[14pt]
\ds{\quad \quad \leq -\frac{\a_{2,1}\lambda_{\bar{q}}}\e\,\Vert \Gamma^\e_2(t)\Vert^{\bar{q}}_{L^{\bar{q}}}+\frac 1\e\langle F_2(t,u^\e(t),\Gamma^\e_2(t) + {W}^{\e}_{A_2}(t)) , \Gamma^\e_2(t)\,|\Gamma^\e_2(t)|^{\bar{q}-2}\rangle.}
\end{array}	
\end{equation}
Since
\[\bar{q}\left(\frac 2{\kappa_1}\wedge \frac{m_1}{m_2}\wedge 1\right)=\left(2\kappa_1 m_1\vee 4 m_2\right)\left(\frac 2{\kappa_1}\wedge \frac{m_1}{m_2}\wedge 1\right)\leq 4 m_1,\]
in view of \eqref{sy620} and \eqref{sy638}, for every $\d>0$ we have
\begin{equation}
\label{sy627}	
\begin{array}{l}
\ds{|\langle F_2(t,u^\e(t),\Gamma^\e_2(t) + {W}^{\e}_{A_2}(t)) , \Gamma^\e_2(t)\,|\Gamma^\e_2(t)|^{\bar{q}-2}\rangle|}\\[14pt]
\ds{\leq L_2\int_{\mathcal{O}}|\Gamma^\e_2(t,\xi)|^{\bar{q}}\,d\xi+c\int_{\mathcal{O}	}|\Gamma^\e_2(t,\xi)|^{\bar{q}-1}\left(1+|u^\e(t,\xi)|^{\frac 2{\kappa_1}\wedge \frac{m_1}{m_2}\wedge 1}+|{W}^{\e}_{A_2}(t,\xi)|\right)\,d\xi}\\[14pt]
\ds{\leq (L_2+\delta)\,\Vert \Gamma^\e_2(t)\Vert_{L^{\bar{q}}}^{\bar{q}}+c_\delta\left(1+\Vert u^\e(t)\Vert_{L^{4 m_1}}^{4 m_1}+\Vert {W}^{\e}_{A_2}(t)\Vert_{L^{\bar{q}}}^{\bar{q}}\right).}
\end{array}
\end{equation}
Hence, if we put together \eqref{sy626} and \eqref{sy627}, 
we get
\[\frac{d}{dt} \|\Gamma^\e_{2}(t)\|^{\bar{q}}_{L^{\bar{q}}}\leq -\frac{1}\e\left(\a_{2, 1}\lambda_{\bar{q}}-(L_2+\delta)\right)\bar{q}\,\|\Gamma^\e_{2}(t)\|^{\bar{q}}_{L^{\bar{q}}}+\frac {c_\d \bar{q}}\e\left(1+\Vert u^\e(t)\Vert_{L^{4 m_1}}^{4 m_1}+\Vert {W}^{\e}_{A_2}(t)\Vert_{L^{\bar{q}}}^{\bar{q}}\right).\]
Now, thanks to \eqref{sy630}
we can fix $\bar{\delta}>0$ such that
\[\vartheta:=\left(\a_{2, 1}\lambda_{\bar{q}}-(L_2+\bar{\delta})\right)\bar{q}>0,\]
and by comparison, we have
\[\mathbb{E}^\epsilon\|\Gamma^\e_{2}(t)\|^{\bar{q}}_{L^{\bar{q}}}\leq e^{-\frac{\vartheta}\e t}\Vert v_0\Vert_{L^{\bar{q}}}^{\bar{q}}+\frac{c}\e\int_0^t e^{-\frac{\vartheta}\e (t-s)}\left(1+\mathbb{E}^\epsilon\Vert u^\e(s)\Vert_{L^{4 m_1}}^{4 m_1}+\mathbb{E}^\epsilon\Vert {W}^{\e}_{A_2}(s)\Vert_{L^{\bar{q}}}^{\bar{q}}\right)\,ds.\] 
As $\vartheta>0$, thanks to \eqref{sy624} we get
\begin{equation}
\label{sy634}
\begin{array}{ll}
\ds{\mathbb{E}^\epsilon\,\|v^\e_{2}(t)\|^{\bar{q}}_{L^{\bar{q}}}} & \ds{\leq e^{-\frac{\vartheta}\e t}\Vert v_0\Vert_{L^{\bar{q}}}^{\bar{q}}+\mathbb{E}^\epsilon\,\Vert W^{\e}_{A_2}(t)\Vert^{\bar{q}}_{L^{\bar{q}}}}\\[14pt]
&\ds{ \quad+c\,\left(1+\sup_{s \in\,[0,t]}\mathbb{E}^\epsilon\,\Vert u^\e(s)\Vert_{L^{4 m_1}}^{4 m_1}+\sup_{s \in\,[0,t]}\mathbb{E}^\epsilon\,\Vert W^{\e}_{A_2}(s)\Vert_{L^{\bar{q}}}^{\bar{q}}\right)}\\[18pt]
&\ds{\leq e^{-\frac{\vartheta}\e t}\Vert v_0\Vert_{L^{\bar{q}}}^{\bar{q}}+c\,\left(1+\sup_{s \in\,[0,t]}\mathbb{E}^\epsilon\,\Vert u^\e(s)\Vert_{L^{4 m_1}}^{4 m_1}\right)+c_T.}
\end{array}
	\end{equation}
Thus, if we plug \eqref{sy634} into \eqref{sy633}, we obtain
\[\sup_{s \in\,[0,t]}\,\mathbb{E}^\epsilon\,\|u^\e(s)\|^{4m_1}_{L^{4m_1}} \leq c_T\left(1+\Vert u_0\Vert^{4m_1}_{L^{4m_1}}+\Vert v_0\Vert_{L^{\bar{q}}}^{\bar{q}}\right)+c_T\int_0^t\sup_{r \in\,[0,s]}\mathbb{E}^\epsilon\,\Vert{u}^{\e}(r)\Vert^{4 m_1}_{L^{4 m_1}}\,ds,\]
and Gronwall's lemma implies
\begin{equation}
\label{sy640}
\sup_{s \in\,[0,t]}\,\mathbb{E}^\epsilon\,\|u^\e(s)\|^{4m_1}_{L^{4m_1}} \leq c_T\left(1+\Vert u_0\Vert^{4m_1}_{L^{4m_1}}+\Vert v_0\Vert_{L^{\bar{q}}}^{\bar{q}}\right).	
\end{equation}
Thanks to \eqref{sy634}, this gives
\begin{equation}
\label{sy641}
\sup_{s \in\,[0,t]}\,\mathbb{E}^\epsilon\,\|v^\e(s)\|^{\bar{q}}_{L^{\bar{q}}} \leq c_T\left(1+\Vert u_0\Vert^{4m_1}_{L^{4m_1}}+\Vert v_0\Vert_{L^{\bar{q}}}^{\bar{q}}\right).	
\end{equation}
Therefore, recalling how $V$ was defined in \eqref{sy120}, from \eqref{sy640} and \eqref{sy641} we obtain \eqref{eq h4}. 
\end{proof}

\begin{Remark}
{\em   When $\kappa_1\leq 2$ and $m_1\geq m_2$, condition \eqref{sy638} is always satisfied, as we are assuming that $g(t,\xi,\cdot)$ has linear growth, uniformly with respect to $(t,\xi) \in\,[0,T]\times \mathcal{O}$.

  }	
\end{Remark}

\subsection{About Hypothesis \ref{h5}}
We need to show that if $\mu^x(dy)$ is the invariant measure of the fast motion with frozen slow component $x$ introduced in Section \ref{fast motion}, then condition \eqref{sy22} is satisfied. Namely we need to show that
\[
\bar{V}(x):=\int_{H} V(x, y)\,\mu^x(dy)<\infty,\ \ \ \ \ x \in \,\Pi_1 D_V.\] 
Recalling how $V$ was defined in \eqref{sy120}, for every $x \in \,L^{4m_1}(\mathcal{O})$, we have
\begin{equation}
\label{sy675}
\bar{V}(x)\leq c_V\int_{H} \left(1+\Vert x\Vert_{L^{4m_1}}^{2 m_1}+\Vert y\Vert_{L^{\bar{q}}}^{\bar{q}/2}\right)\,\mu^x(dy)\leq c_V\left(1+\Vert x\Vert_{L^{4m_1}}^{2 m_1}\right)+c_V\int_{H}\Vert y\Vert_{L^{\bar{q}}}^{\bar{q}/2}\,\mu^x(dy),	
\end{equation}
where $\bar{q}=2\kappa_1 m_1\vee 4m_2$. This means that \eqref{sy22} follows once we show that 
\[
\int_H	\Vert y\Vert_{L^{\bar{q}}}^{\bar{q}/2}\,\mu^x(dy)<\infty,\ \ \ \ \ x \in \,L^{4m_1}(\mathcal{O}).
\]

\begin{Lemma}
Under the same assumptions of Lemma \ref{lemma7.2} we have 
\begin{equation}
\label{sy670}
\int_H	\Vert y\Vert_{L^{\bar{q}}}^{\bar{q}}\,\mu^x(dy)<\infty,\ \ \ \ \ x \in \,L^{4m_1}(\mathcal{O}).
\end{equation}
\end{Lemma}

\begin{proof}
Due to the invariance of $\mu^x$, for every $t\geq 0$, we have
\[\begin{array}{ll}
\ds{\int_H	\Vert y\Vert_{L^{\bar{q}}}^{\bar{q}}\,\mu^x(dy)\leq }  &  \ds{\liminf_{R\to\infty}\int_H	\left(\,\Vert y\Vert_{L^{\bar{q}}}^{\bar{q}}\wedge R\right)\,\mu^x(dy)}\\[14pt]
&
\ds{=\liminf_{R\to\infty}\int_H	P^x_t\left(\,\Vert \cdot\Vert_{L^{\bar{q}}}^{\bar{q}}\wedge R\right)(y)\,\mu^x(dy)\leq \int_H	P^x_t\,(\Vert \cdot\Vert_{L^{\bar{q}}}^{\bar{q}})(y)\,\mu^x(dy),}
\end{array}\]
where $P^x_t$ is the semigroup associated with equation \eqref{eqfrozenslow}.

By proceeding as in the proof of Lemma \ref{lemma7.2},  we have
\[P^x_t\,(\Vert \cdot\Vert_{L^{\bar{q}}}^{\bar{q}})(y)=\mathbb{E}\,\Vert v^{x,y}(t)\Vert^{\bar{q}}_{L^{\bar{q}}}\leq e^{-\vartheta t}\Vert y\Vert_{L^{\bar{q}}}^{\bar{q}}+c\left(1+\Vert x\Vert_{L^{4m_1}}^{4 m_1}\right)+\sup_{s \in\,[0,t]}\mathbb{E}\Vert W_{A_2}(s)\Vert_{L^{\bar{q}}}^{\bar{q}},\]
and this implies
\[\int_H	\Vert y\Vert_{L^{\bar{q}}}^{\bar{q}}\,\mu^x(dy)\leq e^{-\vartheta t}\int_H	\Vert y\Vert_{L^{\bar{q}}}^{\bar{q}}\,\mu^x(dy)+c_t\left(1+\Vert x\Vert_{L^{4m_1}}^{4 m_1}\right).\]
Thus, if we pick $\bar{t}>0$ such that $e^{-\vartheta \bar{t}}\leq 1/2$, we conclude that 
\[\int_H	\Vert y\Vert_{L^{\bar{q}}}^{\bar{q}}\,\mu^x(dy)\leq 2c_{\bar{t}}\left(1+\Vert x\Vert_{L^{4m_1}}^{4 m_1}\right),\]
and \eqref{sy670} follows.

\end{proof}

\subsection{About Condition \eqref{sy20}}
According to \eqref{sy675} and \eqref{sy670}, we have that 
\[\bar{V}(x)\leq c\left(1+\Vert x\Vert_{L^{4m_1}}^{4 m_1}\right),\ \ \ \ \ \ x \in \,L^{4m_1}(\mathcal{O}).\]
Then, condition \eqref{sy20} holds if
\[ \sup_{\e \in (0,1)}\int_0^T\mathbb{E}^\e \Vert u^\e(s)\Vert_{L^{4m_1}}^{4 m_1}  dt <\infty, \]
and this is true because of \eqref{sy640}.

\appendix
\section{{Existence of a martingale solution}}

In what follows, $H$ is a separable Hilbert space, endowed with the scalar product $\langle\cdot,\cdot\rangle$ and the corresponding norm $\Vert\cdot\Vert$. 
We are considering the following equation in the Hilbert space $H$
\begin{equation}
	\label{eq-a}
	dX(t)=\left[A X(t)+F(t,X(t))\right]\,dt+Q\,dW_t,\ \ \ \ \ X(0)=x,
\end{equation}
where $W(t)$, $t\geq 0$,  is a space-time white noise. In what follows, we shall assume that the linear operators $A$ and $Q$ and the non-linearity $F$ 
satisfy the following assumptions.

\begin{Hypothesisp}{A1}\label{h1-app}
\begin{enumerate}
\item
The self-adjoint operator $A$ generates an analytic semigroup $S(t)$.
\item
There exists a complete orthonormal basis $\{e_k\}_{k \in \mathbbm{N}}$ and two sequences of positive real numbers $\{\a_{k}\}_{k \in \mathbb{N}}$ and $\{\lambda_k\}_{k \in \mathbb{N}}$ such that
\begin{equation}
Ae_k = -\a_ke_k,\ \ \ \ \ \ Qe_k=\lambda_{k}e_k,\ \ \ \ \ \ \ k \geq 1.
\end{equation} 
Moreover $\inf_{k \in\,\mathbb{N}}\a_k>0$. 
\item
There exists some $\gamma_0>0$ such that 
\begin{equation}
\sum_{k \in \mathbbm{N}} \lambda_k^2 \a_k^{2\gamma_0-1} < \infty.
\end{equation}
\end{enumerate}
\end{Hypothesisp}

For any $\gamma>0$, the operator $(-A)^{-\gamma}$ is compact and thus the embedding $D((-A)^{\gamma}) \hookrightarrow H$ is compact. Moreover, by the closed graph theorem, the set $\{ x \in\,H\,:\,\|(-A)^{\gamma}x\|  \leq R\}$ is not only precompact but also compact for any $\gamma>0$ and $R>0$. We denote the graph norm of the operator $(-A)^{\gamma}$ by $\|\cdot\|_{\gamma}$.

\begin{Hypothesisp}{A2}\label{h2-app}
There exists a family of measurable and bounded mappings $\{F^\theta\}_{\theta \in (0,1)}$ defined on $[0,T] \times H$ with values in $H$, such that the following conditions hold.
\begin{enumerate}
\item For every $\theta \in (0,1)$ and $h \in\,H$ and for every $t \in\,[0,T]$, the mapping 
\[x \in\,H\mapsto \langle F^\theta(t,x),h\rangle \in\,\mathbb{R}\] is  continuous.
\item
There exists a  convex and lower semi-continuous mapping $V: H \to [1+\infty]$ such that for all $\theta \in (0,1)$
\begin{equation}
\label{app1}
\Vert F^\theta(t,x) \Vert_H^{2}  \leq V(x),\ \ \ \ x \in H,
\end{equation}
and 
\begin{equation}
\label{app18}
\Vert F(t,x)- F^\theta(t,x)\Vert  \leq \theta V(x),\ \ \ \ x \in D(F).
\end{equation}
\end{enumerate}
\end{Hypothesisp}
Here and in what follows, we shall define $\Vert F(t,x)\Vert=+\infty$, whenever $x \notin D(F)$. In particular $D_V:=\left\{ x \in\,H\ :\ V(x)<\infty\right\}\subseteq D(F).$

In what follows, we shall prove the following result.
\begin{Theorem}
	\label{teo-mart}
Under Hypotheses \ref{h1-app} and \ref{h2-app}, and Hypothesis \ref{h3-app} below, for every initial condition $x \in\,D_V$ there exists a martingale solution in $C([0,T];H)$ for problem \eqref{eq-a}. That is, there exist some filtered probability space $(\bar{\Omega}, \bar{\mathcal{F}}, \{\bar{\mathcal{F}}_t\}_{t  \in\,[0,T]},\bar{\mathbb{P}})$, a space-time white noise $\bar{W}_t$, $t \in\,[0,T]$, adapted to the filtration $\{\bar{\mathcal{F}}_t\}_{t  \in\,[0,T]}$, and a $C([0,T];H)$-valued process $\bar{X}$ such that 
\[\bar{X}(t)=S(t)x + \int_0^t S(t-s) F(s,\bar{X}(s,x)) ds + \int_0^t S(t-s) Q\, d\bar{W}_s.\]
\end{Theorem}

\subsection{The approximating problem}

Let $\{\theta_n\}_{n \in \mathbb{N}}\subset (0,1)$ be a sequence which converges to $0$. We denote $F^{\theta_n}$ by $F_n$ for simplicity and consider the following approximating problem
\begin{equation}\label{eq app}
dX_n(t)=\left[AX_n(t) dt + F_n(t, X_n(t)) \right]dt + Q dW_t,\ \ \ \ \ X_n(0)=x \in D(F).
\end{equation}

As a consequence of  Hypotheses \ref{h1-app} and \ref{h2-app}, equation \eqref{eq app} admits a martingale solution in $C([0,T];H)$ (see  \cite{GDGB} and \cite{DP} for a proof). That is, there exist some probability space $(\Omega^n, \mathcal{F}^n, \{\mathcal{F}^n_t\}_{t  \in\,[0,T]},\mathbb{P}^n)$, a space-time white noise $W^n_t$, $t \in\,[0,T]$, adapted to the filtration $\{\mathcal{F}^n_t\}_{t  \in\,[0,T]}$ and a $C([0,T];H)$-valued process $X_n$ such that 
\begin{equation}\label{eq mgs1}
\begin{aligned}
X_n(t)=S(t)x + \int_0^t S(t-s) F_n(s,X_n(s,x)) ds + W_A^n(t),
\end{aligned}
\end{equation}
where $W_A^n(t)$ is the stochastic convolution 
\begin{equation}
W_A^n(t): = \int_0^t S(t-s) Q\, dW^n_s.
\end{equation}
Equivalently, $X_n$ satisfies the following equation \begin{equation}\label{eq1.2}
\begin{aligned}
\langle X_{n}(t), \xi(t) \rangle= &\langle x, \xi(0) \rangle +  \int_0^t \langle X_{n}(s) ,  \xi'(s)+A^*\xi(s) \rangle\, ds\\[14pt]
& +  \int_0^t \langle F_{n}(s, X_{n}(s)) ,  \xi(s) \rangle\, ds +  \int_0^t \langle Q dW^n_s, \xi(s) \rangle,
\end{aligned}
\end{equation}
for any $\xi \in C^1([0,T];D(A^*))$.

\begin{Hypothesisp}{A3}\label{h3-app}
For every $x \in\,D(F)$, we have
\begin{equation}\label{eqh3}
\sup_{n \in \mathbb{N}}  \int_0^T\mathbb{E}^n V(X_n(t))\,dt  \leq c\,V(x)
\end{equation}
\end{Hypothesisp}

Notice that, as a consequence of Hypothesis \ref{h3-app}, we have that  $\mathbb{P}^n(X_n(t) \in D(F),\ t \in\,[0,T])=1$, for all $n \in \mathbb{N}$.
Moreover, as a consequence of Hypothesis \ref{h1-app}, it is possible to show that for every $p\geq 1$ and $\gamma \in [0,\gamma_0)$
\begin{equation}
\label{app2}
\sup_{n \in \mathbb{N}}\,\mathbb{E}^n \sup_{t \in [0,T]} \|W^n_A(t)\|_\gamma^p < \infty.
\end{equation}
Moreover,
\begin{equation}
\label{app3}
\sup_{n \in\,\mathbb{N}}\,\mathbb{E}^n \|W_A^n(t+h) - W_A^n(t)\|_\gamma^p \leq c\,h^{p\gamma_\star},
\end{equation}
where $\gamma_\star:=(\gamma_0-\gamma) \wedge 1/2$.
In particular, from the Garsia-Rademich-Rumsey Theorem we obtain that for every $\gamma \in [0,\gamma_0)$ there exists some $\beta>0$ such that
\begin{equation}
\label{app5}
\sup_{n \in\,\mathbb{N}}	\mathbb{E}^n\Vert W^n_A\Vert^2_{C^\beta([0,T];D((-A)^\gamma))}<\infty.
\end{equation}
For all details see e.g. \cite[Section 5.4]{DP}.

\medskip

Now, we define
\[
\Psi_{n} (t):= \int_0^t S(t-s) F_{n}(s, X_{n}(s)) ds.
\]
\begin{Lemma}\label{l0.3}
Let Hypotheses \ref{h1-app}, \ref{h2-app} and \ref{h3-app} hold and assume that $\gamma \in [0,1/2) $. Then
\begin{equation}
\label{app6}
\sup_{n \in \mathbb{N}}\mathbb{E}^n \sup_{t \in [0,T]}\|\Psi_n(t)\|_\gamma^{2} \leq c_T  V(x),\ \ \ \ \ x \in\,D_V.
\end{equation}
\end{Lemma}
\begin{proof}
We have
\[
\begin{aligned}
\mathbb{E}^n \sup_{t \in [0,T]}\|\Psi_n(t)\|_\gamma^{2} \leq&\  c \sup_{t \in [0,T]} \left(\int_0^t (t-s)^{-2\gamma } \,ds \right)^{} \mathbb{E}^n\sup_{t \in [0,T]}\int_0^t \|F_{n}(s, X_{n}(s))\|^{2} ds^{}\\
\leq& \,c_T \, \mathbb{E}^n \int_0^T \|F_{n}(s, X_{n}(s))\|^{2} ds.
\end{aligned}
\]
Then \eqref{app6} follows from \eqref{app1}    and \eqref{eqh3}. 
\end{proof}
As for  the time continuity for $\Psi_n$, we have the following result.

\begin{Lemma}\label{l0.4}
Let Hypotheses \ref{h1-app}, \ref{h2-app} and \ref{h3-app} hold and fix $\gamma \in [0, 1/2)$. Then there exists $\beta>0$ such that 
\[
\sup_{n \in \mathbb{N}}\mathbb{E}^n \left[\sup_{\delta\in(0,1)}\sup_{t, s\in [0,T],\ |t-s|\leq \delta}\left(\frac{\|(\Psi_n(t)- \Psi_n(s))\|_\gamma}{|t-s|^{\beta}}\right)^2\right] \leq c\,V(x),\ \ \ \ \ x \in\,D_V.
\]
\end{Lemma}
\begin{proof}
For every $\gamma \in (0,1/2)$, $h \in\,(0,\delta)$ and $t \in\,[0,T-\delta]$, we have
\[
\begin{array}{ll}
\ds{\|\Psi_n(t+h)- \Psi_n(t) \|_\gamma\leq }  &  \ds{\,c\, \int_t^{t+h}\left\| S(t+h-s) F_{n}(s, X_{n}(s))\right\|_\gamma\, ds}\\[14pt]
&\ds{+ c\, \int_0^{t} \|(S(t+h-s)-S(t-s)) F_{n}(s, X_{n}(s))\|_\gamma\, ds,}\end{array}
\]
so that
\[\begin{array}{l}
\ds{\|(\Psi_n(t+h)- \Psi_n(t)) \|_\gamma^2\leq c \int_t^{t+h} (t+h-s)^{-2 \gamma }\, ds \int_t^{t+h} \|F_{n}(s,X_{n}(s))\|^{2} ds}\\[14pt]
\ds{+c \left(\int_0^{t} \int_{t-s}^{t+h-s}\|AS(u) du\ F_{n}(s, X_{n}(s))\|_\gamma\,ds\right)^{2}}\\[14pt]
\ds{\leq c\int_t^{t+h} (t+h-s)^{-2\gamma} ds \int_0^{T}  \|F_{n}(s,X_{n}(s))\|^{2} ds}\\[14pt]
\ds{+c \int_0^t \left(\int_{t-s}^{t+h-s}\|(-A)^{\gamma+1} S(r)\|_{\mathcal{L}(H)}\,dr\right)^2 ds \int_0^T \|F_n(s,X_n(s))\|^{2}\,ds.}
	\end{array}\]
Therefore, in view of Hypothesis \ref{h3-app}, we have
\[
\begin{array}{l}
\ds{\mathbb{E}^n\left(\sup_{\delta\in(0,1)}\sup_{t\in [0,T-\delta]}\sup_{h \in (0,\delta)}\|(\Psi_n(t+h)- \Psi_n(t))\|_\gamma^{2} h^{-2\beta}\right)}\\[14pt]
\ds{\leq \, c_T\sup_{\delta\in(0,1)}\sup_{t\in [0,T-\delta]}\sup_{h \in (0,\delta)}h^{-2\beta}\int_t^{t+h} (t+h-s)^{-2\gamma}\,ds^{}\,V(x)}\\[14pt]
\ds{+ c_T\sup_{h_0\in(0,1)}\sup_{t\in [0,T-h_0]}\sup_{h \in (0,h_0)} h^{-2 \beta}\int_0^t \left(\int_{t-s}^{t+h-s}r^{-(\gamma+1) }\, dr\right)^{2} ds^{}\,V(x).}
	\end{array}
\]
Since 
\[\big((t-s)^{-\gamma}-(t+h-s)^{-\gamma })^{2} \leq  (t-s)^{-2\gamma} - (t+h-s)^{-2\gamma  },\]
 we have
\[
\begin{array}{l}
\ds{\mathbb{E}^n\left(\sup_{\delta\in(0,1)}\sup_{t\in [0,T-\delta]}\sup_{h \in (0,\delta)}\|(\Psi_n(t+h)- \Psi_n(t)) \|_\gamma^{2}\,h^{-2\eta}\right)}\\[18pt]
\ds{\leq  c_T\,V(x)\sup_{h \in (0,1)}\left(h^{-2\beta } h^{-2\gamma +1}  +  h^{-2\beta }\int_0^t  \left((t+h-s)^{-2\gamma  }- (t-s)^{-2\gamma }\right) ds \right).}
\end{array}
\]
Moreover, since $1-2\gamma<1$, we have 
\[(t+h)^{-2\gamma +1}- t^{-2\gamma +1} \leq h^{-2\gamma +1}.\] Therefore, if we take $\beta<1/2-\gamma$, we obtain
\[
\begin{array}{l}
\ds{\mathbb{E}^n\left(\sup_{h_0\in(0,1)}\sup_{t\in [0,T-h_0]}\sup_{h \in (0,h_0)}\|(\Psi_n(t+h)- \Psi_n(t)) \|_\gamma^{2}\,h^{-2\beta}\right)}\\[18pt]
\ds{\leq  c_T\,V(x)\sup_{h\in(0,1)}h^{1-2\gamma-2\beta}\leq c_T\,V(x).}
\end{array}
\]
Notice that when $\gamma =0$, the result holds for all $\beta \in (0,1/2)$.
\end{proof}

\medskip

For every $n \in\,\mathbb{N}$, we define
\[y_n(t):=\Psi_n(t) + W_A^n(t),\ \ \ \ t \in [0,T].\]
As a consequence of \eqref{app5} and Lemmas \ref{l0.3} and  \ref{l0.4}, there exist  $\gamma, \eta>0$ such that
\begin{equation} \label{app10}
\sup_{n \in \mathbb{N}}\,\mathbb{E}^n \sup_{t \in [0,T]} \|y_{n}(t)\|_{\gamma}^{2} <\infty,
\end{equation}
and 
\begin{equation}  \label{app11}
\sup_{n \in \mathbb{N}}\,\mathbb{E}^n \sup_{t,s, \in [0,T]}\left(\frac{\|y_{n}(t)-y_{n}(s)\|} {|t-s|^{\beta}}\right)^2<\infty.
\end{equation}
This implies the following tightness result.
\begin{Lemma}\label{l4.1.6}
Under Hypotheses \ref{h1-app}, \ref{h2-app} and \ref{h3-app}, for every $x \in\,D_V$ the family of measures $\{\mathcal{L}( X_{n})\}_{n \in \mathbb{N}}$ is tight in   $C([0,T];H))$. 
\end{Lemma}
\begin{proof}
Let $\gamma$ and $\eta$ be the same as the ones in \eqref{app10} and \eqref{app11}. For every $R>0$, we introduce the set
\[
K_R:= \{ u \in C([0,T];H)\ :\ \Vert u \Vert_{C([0,T];D((-A)^\gamma))}+\sup_{t,s\in [0,T]} \frac{\|u(t)- u(s)\Vert}{|t-s|^{\beta}} \leq R \}.
\]
By Ascoli-Arzela theorem, $K_R$ is  compact in $C([0,T];H)$. Moreover, according to \eqref{app10} and \eqref{app11}, we have
\[
\sup_{n \in \mathbb{N}}\,\mathbb{P}^{n}\left(y_n\in K_R^c\right)\leq c_T R^{-1}.
\]
Then, for every $\eta>0$ we can find $R_\eta>0$ such that
\[
\sup_{n \in \mathbb{N}}\,\mathbb{P}^{n}\left(y_n\in K_{R_\eta}^c\right)\leq \eta.
\]
Since $X_n= y_n+ S(\cdot)x$, we have 
\[
\mathbb{P}^{n}\left(X_n \in K_M^c + S(\cdot)x\right)=\mathbb{P}^{n}\left(y_n\in K_M^c\right)\leq \eta.
\]
Therefore,  since $K_{R_\eta}+S(\cdot)x$ is a compact set in $C([0,T];H)$, due to the arbitrariness of $\eta$ we  conclude that the family of measures $\{ \mathcal{L}(X_{n})\}_{n \in \mathbb{N}}$ is tight in $C([0,T];H)$. 
\end{proof}

\subsection{Construction of a martingale solution}\label{existence}

In view of  Lemma \ref{l4.1.6}, the sequence $\{(X_n, W_n)\}_{n \in\,\mathbb{N}}\subset C([0,T];H)\times C([0,T];\mathcal{D}^\prime)$ is tight. Due to the Prokhorov theorem, there exists a subsequence, still denoted by $\{(X_n, W_n)\}_{n \in\,\mathbb{N}}$, that converges weakly to some measure $\mu$ on $C([0,T];H)\times C([0,T];\mathcal{D}^\prime)$. By Skorohod's theorem we can find a probability space $(\bar{\Omega}, \bar{\mathcal{F}}, \bar{\mathbb{P}})$ and a sequence of random variables $\{(\bar{X}_{n}, \bar{W}^n)\}_{n \in \mathbb{N}}$, and $(\bar{X},\bar{W})$ with values in $C([0,T];H)\times C([0,T];\mathcal{D}^\prime)$, such that 
\begin{equation} \label{app12}
\mathcal{L}(\bar{X}_n, \bar{W}^n)= \mathcal{L}({X}_n, W^n),\ \ \ \ \ \ \mathcal{L}(\bar{X}, \bar{W})=\mu,
\end{equation} 
and \begin{equation}
 \label{app16}
 \lim_{n\to\infty} (\bar{X}_n,\bar{W}^n)=(\bar{X}, \bar{W}),\ \ \ \ \bar{\mathbb{P}}-\text{a.s. in}\ C([0,T];H)\times C([0,T];\mathcal{D}^\prime).
 \end{equation}
 Notice that since $\mathcal{L}(\bar{W}^n)=\mathcal{L}({W}^n)$ and $W^n$ is a space-time white noise, for every $n \in\,\mathbb{N}$, then $\bar{W}$ is a space-time white noise.

 For any  $\xi\in C^1([0,T]; D(A^*))$ and $t \in\,[0,T]$, we  define
\begin{equation}
\label{app17}
\begin{aligned}
\bar{W}_{\xi}^n(t):= \langle \bar{X}_n(t), &\xi(t)\rangle- \langle  x, \xi(0)\rangle-\int_0^{t} \langle \bar{X}_n(s), \xi'(s)+A^*\xi(s) \rangle ds \\
&- \int_0^{t} \langle  F_n(s, \bar{X}_n(s)) , \xi(s) \rangle ds.
\end{aligned}
\end{equation}
Recalling that $X_n$ satisfies equation \eqref{eq1.2}, thanks to \eqref{app12}, we have
\[
\bar{W}_{\xi}^n(t)=\int_0^t \langle Q d\bar{W}^n_s, \xi(s)\rangle,\ \ \ \ \ t \in\,[0,T].
\]
Hence,
thanks to \eqref{app16}, we have
\begin{equation}
\label{app20}
\lim_{n\to\infty} \bar{W}_{\xi}^n=\int_0^t \langle Q d\bar{W}_s, \xi(s)\rangle,\ \ \ \ \ \text{in}\ C([0,T]),\ \ \ \ \bar{\mathbb{P}}-\text{a.s.}	\end{equation}

Therefore, if we define 
$\bar{\mathcal{F}}_t:=\sigma (\bar{W}_s,\ s\leq t)$, for $t \in\,[0,T]$, we conclude that 
\[(\bar{\Omega}, \bar{\mathcal{F}}, \{\bar{\mathcal{F}}_t\}_{t \in\,[0,T]},\bar{\mathbb{P}}, \bar{W}, \bar{X})\] is a martingale solution for equation \eqref{eq-a}, once we prove the following result.

\begin{Lemma}
\label{lemma4.6}
There exists a subsequence of $\{ \bar{W}_{\xi}^n\}_{n\in \mathbb{N}}$,  still indexed by $n$,  such that 
\begin{equation}\label{eq4.1}
\begin{aligned}
\lim_{n \to \infty} \bar{W}_{\xi}^n= \langle \bar{X}, &\xi\rangle-\langle x, \xi(0) \rangle\\
&-\int_0^\cdot \langle \bar{X}(t), \xi'(s)+A^*\xi(s) \rangle ds - \int_0^\cdot \langle  F(s, \bar{X}(s)) , \xi(s) \rangle ds,
\end{aligned}
\end{equation}
in $C([0,T];H)$,  $\bar{\mathbb{P}}$-almost surely.
\end{Lemma}
\begin{proof}
Due to \eqref{app16}, we only need to prove the convergence  of the fourth term on the right hand side  of equation \eqref{app17}.
We have
\[
\begin{aligned}
\liminf_{n \to \infty}\,\bar{\mathbb{E}} \sup_{t \in [0,T]}\left| \int_0^t \langle  F_n(s,\bar{X}_{n}(s)), \xi(s) \rangle ds-   \int_0^t \langle  F_{}(s,\bar{X}_{}(s)), \xi(s) \rangle \,ds\right|\\
 \leq  \liminf_{n \to \infty}\, \bar{\mathbb{E}} \int_0^T |\langle F_n(s,\bar{X}_{n}(s))- F(s,\bar{X}_{}(s)), \xi(s)\rangle|  \,ds. \\
\end{aligned}
\]
Due to \eqref{app18}, for every $n \in\,\mathbb{N}$ and $s \in\,[0,T]$, we have
\[\begin{array}{l}
\ds{|\langle F_n(s,\bar{X}_{n}(s))- F(s,\bar{X}_{}(s)), \xi(s)\rangle|}	\\[14pt]
\ds{\leq |\langle F_n(s,\bar{X}_{n}(s))- F(s,\bar{X}_{n}(s)), \xi(s)\rangle|+|\langle F(s,\bar{X}_{n}(s))- F(s,\bar{X}_{}(s)), \xi(s)\rangle|}\\[14pt]
\ds{\leq c_\xi \,\theta_n \,V(\bar{X}_n(s))+|\langle F(s,\bar{X}_{n}(s))- F(s,\bar{X}_{}(s)), \xi(s)\rangle|.}
\end{array}\]
Therefore, thanks to \eqref{eqh3} and \eqref{app12}, 
\begin{equation}  \label{app19}
\begin{array}{l}
\ds{\liminf_{n \to \infty}\bar{\mathbb{E}} \sup_{t \in [0,T]}\left| \int_0^t \langle  F_n(s,\bar{X}_{n}(s)), \xi(s) \rangle ds-   \int_0^t \langle  F_{}(s,\bar{X}_{}(s)), \xi(s) \rangle ds\right|}\\[14pt]
\ds{\leq c_\xi\,\theta_n \,V(x) 
  +  \liminf_{n \to \infty}\, \bar{\mathbb{E}}\int_0^T \left| \langle F(s,\bar{X}_{n}(s))- F(s,\bar{X}_{}(s)),\xi(s)\rangle \right| \,ds .}
\end{array}
\end{equation}

Now, for every $n,m \in\,\mathbb{N}$ and $s \in\,[0,T]$, we have
\[\begin{array}{l}
\ds{\langle F(s,\bar{X}_{n}(s))- F(s,\bar{X}_{}(s)),\xi(s)\rangle =\langle F(s,\bar{X}_{n}(s))- F_m(s,\bar{X}_{n}(s)),\xi(s)\rangle}\\[14pt]
\ds{ +\langle F_m(s,\bar{X}_{n}(s))- F_m(s,\bar{X}_{}(s)),\xi(s)\rangle +\langle F_m(s,\bar{X}(s))- F(s,\bar{X}_{}(s)),\xi(s)\rangle, }	
\end{array}\]
so that, thanks to \eqref{app18} 
\begin{equation}  \label{app21} \begin{array}{l}
\ds{\int_0^T\left|\langle F(s,\bar{X}_{n}(s))- F(s,\bar{X}_{}(s)),\xi(s)\rangle\right|\,ds\leq c_\xi\,\theta_m\left(\int_0^T V(\bar{X}_n(s))\,ds+\int_0^T V(\bar{X}(s))\,ds\right)}\\[14pt]
\ds{+\int_0^T\left|\langle F_m(s,\bar{X}_{n}(s))- F_m(s,\bar{X}_{}(s)),\xi(s)\rangle\right|\,ds.}	
\end{array}\end{equation}
In view of Hypothesis \ref{h3-app}, 
for every $t \in [0,T]$ and $x \in D_V$, we have
\begin{equation}
\label{app15}
\int_0^T\bar{\mathbb{E}}\, V(\bar{X}(t))\,dt  \leq c\, V(x).
\end{equation}
Actually, since $V$ is lower semicontinuous, for every $t \in\,[0,T]$ we have 
\[
V(\bar{X}(t)) = V(\liminf_{n \to \infty} \bar{X}_n(t)) \leq \liminf_{n \to \infty} V(\bar{X}_n(t)).
\]
Therefore, 
\[
\int_0^T\bar{\mathbb{E}}\,V(\bar{X}(t))\,dt  \leq \liminf_{n \to \infty} \int_0^T\bar{\mathbb{E}}\, V(\bar{X}_n(t))\,dt=\liminf_{n \to \infty} \int_0^T{\mathbb{E}^n}\, V({X}_n(t))\,dt,
\]
and thanks to Hypothesis \ref{h3-app}, this implies \eqref{app15}.

Now, \eqref{app15},  together with \eqref{app21} and \eqref{app19}, imply
\[
\begin{array}{l}
\ds{\liminf_{n \to \infty}\bar{\mathbb{E}} \sup_{t \in [0,T]}\left| \int_0^t \langle  F_n(s,\bar{X}_{n}(s)), \xi(s) \rangle ds-   \int_0^t \langle  F_{}(s,\bar{X}_{}(s)), \xi(s) \rangle ds\right|}\\[14pt]
\ds{\leq c_\xi\,(\theta_n +\theta_m)\,V(x) 
  +  \liminf_{n \to \infty} \bar{\mathbb{E}}\int_0^T \left| \langle F_m(s,\bar{X}_{n}(s))- F_m(s,\bar{X}_{}(s)),\xi(s)\rangle \right| \,ds.}
\end{array}
\]
If we fix an arbitrary $\eta>0$, we can find $m_\eta \in\,\mathbb{N}$ such that $c_\xi \theta_{m_\eta} V(x)<\eta$, so that
\[
\begin{array}{l}
\ds{\liminf_{n \to \infty}\bar{\mathbb{E}} \sup_{t \in [0,T]}\left| \int_0^t \langle  F_n(s,\bar{X}_{n}(s)), \xi(s) \rangle ds-   \int_0^t \langle  F_{}(s,\bar{X}_{}(s)), \xi(s) \rangle ds\right|}\\[14pt]
\ds{\leq \eta + c_\xi\,\theta_n\,V(x) 
  +  \liminf_{n \to \infty} \bar{\mathbb{E}}\int_0^T \left| \langle F_{m_\eta}(s,\bar{X}_{n}(s))- F_{m_\eta}(s,\bar{X}_{}(s)),\xi(s)\rangle \right| \,ds.}
\end{array}
\]
Since $F_{m_\eta}(t,\cdot):H\to H$ is strong-weak continuous and bounded and \eqref{app16} holds, we can take the limit above, as $n\to\infty$, and, due to the arbitrariness of $\eta$ we conclude 
\[\liminf_{n \to \infty}\bar{\mathbb{E}} \sup_{t \in [0,T]}\left| \int_0^t \langle  F_n(s,\bar{X}_{n}(s)), \xi(s) \rangle ds-   \int_0^t \langle  F_{}(s,\bar{X}_{}(s)), \xi(s) \rangle ds\right|=0.\]

Finally, by extracting possibly another subsequence again,  we have
\begin{equation}
\lim_{n \to \infty} \sup_{t \in [0,T]}| \int_0^t \langle F_{n}(s,\bar{X}_{n}(s)), \xi(s) \rangle ds-   \int_0^t \langle F_{}(s,\bar{X}_{}(s)), \xi(s) \rangle ds|=0,
\end{equation}
 $\bar{\mathbb{P}}-$almost surely, and our lemma  follows.
\end{proof}

\end{document}